\documentclass[12pt]{amsart}
\usepackage{amsmath,amssymb,amsxtra,anysize,euscript,bbm,combelow}

\usepackage{etex}

\usepackage{tikz, tikz-cd}
\usetikzlibrary{matrix, arrows, calc}

\RequirePackage[pdftex,
colorlinks = true,
urlcolor = blue, 
citecolor = blue, 
linkcolor = blue, 
]
{hyperref}

\newtheorem{thm}{Theorem}[section]
\newtheorem{lemma}[thm]{Lemma}
\newtheorem{theorem}[thm]{Theorem}
\newtheorem{observation}[thm]{Observation}

\newtheorem{prop}[thm]{Proposition}
\newtheorem{proposition}[thm]{Proposition}

\newtheorem{corollary}[thm]{Corollary}

\newtheorem{conjecture}[thm]{Conjecture}

\newtheorem*{prop*}{Proposition}
\newtheorem*{lemma*}{Lemma}

\theoremstyle{definition}

\newtheorem{definition}[thm]{Definition}

\newtheorem{notation}[thm]{Notation}

\newtheorem{example}[thm]{Example}

\theoremstyle{remark}
\newtheorem{remark}[thm]{Remark}

\numberwithin{equation}{section}

\renewcommand{\setminus}{\smallsetminus}

\newcommand{\CO}{\mathcal{O}}

\newcommand{\Hom}{\operatorname{Hom}}
\newcommand{\Hilb}{\operatorname{Hilb}}
\newcommand{\Tr}{\operatorname{Tr}}
\newcommand{\End}{\operatorname{End}}

\newcommand{\SBim}{\mathbb{S}\mathrm{Bim}}
\newcommand{\SMod}{\mathbb{S}\mathrm{Mod}}
\renewcommand{\d}{\delta}
\renewcommand{\a}{\alpha}
\newcommand{\one}{\mathbbm{1}}
\newcommand{\AS}{\EuScript A}
\newcommand{\BS}{\EuScript B}
\newcommand{\CS}{\EuScript C}
\newcommand{\MS}{\EuScript M}
\newcommand{\NS}{\EuScript N}
\newcommand{\PB}{\mathbf{P}}
\newcommand{\QB}{\mathbf{Q}}

\newcommand{\HM}{\mathbb{H}}
\newcommand{\OC}{\mathcal{O}}
\newcommand{\gl}{\mathfrak{gl}}
\renewcommand{\b}{\beta}
\newcommand{\Br}{\operatorname{Br}}
\newcommand{\inv}{^{-1}}
\newcommand{\Id}{\operatorname{Id}}
\newcommand{\HH}{\operatorname{HH}}
\newcommand{\HHH}{\operatorname{HHH}}
\newcommand{\FT}{\operatorname{FT}}
\newcommand{\HT}{\operatorname{HT}}
\newcommand{\KC}{\mathcal{K}}
\newcommand{\Z}{\mathbb{Z}}
\newcommand{\CL}{\mathcal{L}}
\newcommand{\e}{\varepsilon}

\newcommand{\Cone}{\operatorname{Cone}}
\newcommand{\XB}{\mathbf{X}}

\newcommand{\EB}{\mathbf{E}}

\newcommand{\Homc}{\underline{\Hom}}

\newcommand{\pTr}{\pi}
\newcommand{\Q}{\mathbb{Q}}
\newcommand{\C}{\mathbb{C}}
\newcommand{\TS}{\EuScript T}
\newcommand{\US}{\EuScript U}
\renewcommand{\k}{\mathbbm{k}}
\newcommand{\ip}[1]{\left\langle #1\right\rangle}

\newcommand{\frakh}{\mathfrak{h}}

\newcommand{\rouq}{F}

\RequirePackage{color}

\title{Serre duality for Khovanov-Rozansky homology}
\author{Eugene Gorsky}
\address{Department of Mathematics, University of California, Davis}
\address{International Laboratory of Representation Theory and Mathematical Physics, NRU-HSE, Moscow, Russia}
\email{egorskiy@math.ucdavis.edu}
\author{Matthew Hogancamp}
\address{Department of Mathematics, University of Southern California}
\email{hogancam@usc.edu}
\author{Anton Mellit}
\address{Faculty of Mathematics, University of Vienna}
\email{anton.mellit@univie.ac.at}
\author{Keita Nakagane}
\address{Department of Mathematics,
Tokyo Institute of Technology}
\email{nakagane.k.aa@m.titech.ac.jp}

\begin{document}

\begin{abstract}
We prove that the full twist is a Serre functor in the homotopy category of  type A Soergel bimodules.
As a consequence, we relate the top and bottom Hochschild degrees in Khovanov-Rozansky homology,
categorifying a theorem of K\'alm\'an.
\end{abstract}

\maketitle

\setcounter{tocdepth}{1}
\tableofcontents

\section{Introduction}

The category of Soergel bimodules is a categorification of the Hecke algebra. It can be defined for any Coxeter group, but here we focus on type A only, where the corresponding group is $S_n$ and the category of Soergel bimodules will be denoted by $\SBim_n$. Given a braid on $n$ strands, Rouquier \cite{Rouquier} constructed a complex of bimodules in $\SBim_n$ and proved that it is unique up to a canonical homotopy equivalence. Khovanov and Rozansky \cite{Kh07,KR2} used Rouquier complexes to define {\em Khovanov-Rozansky homology} $\HHH$, a categorification of the HOMFLY-PT polynomial.

In recent years, the Rouquier complex for the full twist braid $\FT_n$ has attracted a lot of attention. Elias and the second author \cite{EH2} proved that $\FT_n$ is in the Drinfeld center of the homotopy category of Soergel bimodules $\KC^b(\SBim_n)$.  They also computed the Khovanov-Rozansky homology of the full twist \cite{EH1} and the categorified eigenvalues of $\FT_n$ acting on  $\KC^b(\SBim_n)$. The work of the first and second author, Negu\cb{t} and Rasmussen \cite{GH,GNR} related $\FT_n$ to a natural line bundle $\CO(1)$ on the isospectral Hilbert scheme $X_n$.

In this paper, we prove that $\FT_n$ (or rather its inverse $\FT_n^{-1}$) acts as a kind of 
Serre functor \cite{BK} in  $\KC^b(\SBim_n)$.  Let $\k$ be a field of characteristic $\neq 2$, and set $R:=\k[x_1,\ldots,x_n]$.  We will consider Soergel bimodules over $\k$.  Given a complex of free $R$-modules $X$, we denote by $X^{\vee}=\Homc(X,R)$ the dual complex. Note that the cohomology of $X^{\vee}$ and of $X$ are, in general, related by the universal coefficient spectral sequence which can be rather complicated.

\begin{theorem}
\label{thm:introserre}
For any two complexes $A,B\in \KC^b(\SBim_n)$ one has
$$
\Homc(A,B)\simeq \Homc(\FT_n\otimes B,A)^{\vee}=\Homc(B,\FT_n^{-1}\otimes A)^{\vee}.
$$
\end{theorem}

Here $\Homc(-,-)$ denotes the complex of homs; in the category of complexes of Soergel bimodules, $\Homc(A,B)$ is a $\Z\times \Z$-graded complex of $(R,R)$-bimodules.  Theorem \ref{thm:introserre} is true whether we regard $\Homc(-,-)$ as complexes of right $R$-modules or left $R$-modules.

\begin{remark}
Typically one states Serre duality in the context of categories which are linear over a field $\k$.  The statement above differs from this typical situation in several ways.  First, our category is monoidal, and the duality is taken with respect to the ring $R \cong \End(\one)$ instead of a field.  Second, the morphism spaces are bimodules over this ring, and one may take the duals with respect to either the left or right actions.   Finally, the Serre duality functor itself is tensoring with an object of the category.
\end{remark}

\begin{remark}
In \cite{Bez, MS} it was proven that action of the full twist on the BGG category $\CO$ is the Serre functor (see \S \ref{sec:smod}), which holds in more general types.  We expect that our result also generalizes to other types, though we do not consider this here.
\end{remark}



\begin{remark}
Theorem \ref{thm:introserre} can be compared with a result of Haiman \cite{Haiman} which states that the isospectral Hilbert scheme $X_n$ is Gorenstein with the canonical sheaf $\CO(-1)$, so tensor multiplication by $\CO(-1)$ is a Serre functor. 
\end{remark}

\subsection{A reformulation}
\label{subsec:introreform}
It is much more convenient to restate Theorem \ref{thm:introserre} in a more canonical form.  First, let $\SBim_{1,\ldots,1}\subset \SBim_n$ denote the full subcategory consisting of direct sums of shifted copies of the trivial bimodule $\one= R$.  The inclusion $\SBim_{1,\ldots,1} \rightarrow \SBim_n$ has left and right adjoints $\Pi_L,\Pi_R:\SBim_n\rightarrow \SBim_{1,\ldots,1}$ defined as follows.  The left adjoint $\Pi_L(M)=\HH_0(M)$ is the quotient of $M$ by the sub-bimodule of commutators $fm-mf$, for all $f\in R$ and all $m\in M$, while the right adjoint $\Pi_R(M)=\HH^0(M)$ is the sub-bimodule consisting of elements $m\in M$ with $fm-mf=0$ for all $f\in R$.  When $M$ is a Soergel bimodule, $\HH_0(M)$ and $\HH^0(M)$ are free $R$-modules, hence can be regarded as objects of $\SBim_{1,\ldots,1}$ (see \S \ref{sec:background}).

The additive functors $\HH_0$ and $\HH^0$ can be extended to complexes, and it is not hard to see that
\[
\HH^0(X) \cong \Homc_{R,R}(R,X),\qquad \HH_0(X) \cong \Homc_{R,R}(X,R)^{\vee},
\]
naturally in $X\in \KC^b(\SBim_n)$.  The second of these uses properties of Soergel bimodules.  Thus, Theorem \ref{thm:introserre} has the following as a special case (set $A=R$ and $B=X$).

\begin{theorem}\label{thm:introReform}
For any complex $X\in \KC^b(\SBim_n)$ we have $\HH^0(X)\simeq \HH_0(\FT\otimes X)$ in $\KC^b(\SBim_{1,\ldots,1})$.
\end{theorem}

Using the rigid monoidal structure on $\KC^b(\SBim_n)$ is not hard to see that in fact Theorem \ref{thm:introserre} is equivalent to Theorem \ref{thm:introReform}.  However, the latter is often preferable because the $R$-action is now canonically defined (the left and right $R$-actions on $\HH^0(-)$ and $\HH_0(-)$ coincide).  More importantly, the latter theorem generalizes to a relative version, which we discuss next.

For each subset $I\subset \{1,\ldots,n-1\}$, let $\SBim_I\subset \SBim_n$ denote the full monoidal subcategory generated by the Bott-Samelson bimodules $B_s$ with $s\in I$.  Alternatively, the subgroup of $S_n$ generated by $s\in I$ is of the form $S_{k_1}\times \cdots\times S_{k_r}\subset S_n$, and we will write
\[
\SBim_I =: \SBim_{k_1,\ldots,k_r}
\]
by abuse.

The inclusion $\SBim_{n-1,1}\rightarrow \SBim_n$ has left and right adjoints $\pi_L,\pi_R:\SBim_n\rightarrow \SBim_{n-1,1}$ defined as follows.  With respect to the identification $R=\k[x_1,\ldots,x_n]$,  $\pi_L(M)$ and $\pi_R(M)$ are the cokernel and kernel of $x_n\otimes 1 -1\otimes x_n$ acting on $M$, respectively.

\begin{remark}
In the main body of the paper, we write $\pi^+=\pi_L$ and $\pi^-=\pi_R$ because of the interaction of these functors with positive and negative Rouquier complexes.
\end{remark}

\begin{theorem}\label{thm:introrelative}
Let $\CL_n:=\FT_n\otimes \FT_{n-1}\inv$ denote the Rouquier complex of the Jucys-Murphy braid $\sigma_{n-1}\cdots\sigma_2\sigma_1^2\sigma_2\cdots \sigma_{n-1}$.  For each complex $X\in \KC^b(\SBim_n)$ we have
\[
\pi_R(X)\simeq \pi_L(\CL_n\otimes X),
\]
naturally in $X$.
\end{theorem}
By analogy with Theorem \ref{thm:introReform}, we may refer to $\CL_n\inv$ as the \emph{relative Serre functor} for $\SBim_n$ relative to $\SBim_{n-1,1}$.  In relation to the conjectures in \cite{GNR}, this result is a monoidal, algebraic analogue of a geometric statement regarding the Hilbert scheme of points $\Hilb^n(\C^2)$ relative to the nested Hilbert scheme $\Hilb^{n-1,1}(\C^2)$. More precisely, $\Hilb^{n-1,1}(\C^2)$ yields a smooth correspondence between 
$\Hilb^n(\C^2)$ and $\Hilb^{n-1}(\C^2)$, and the analogues of $\pi_R$ and $\pi_L$ differ by the canonical line bundle 
on $\Hilb^{n-1,1}(\C^2)$ which was computed e.g. in \cite[Proposition 3.6.4]{Haiman} and corresponds to $\CL_n^{-1}$. 

We expect that this statement generalizes to arbitrary Coxeter systems in the following way.  Let $(W,S)$ be any finite Coxeter system with longest element $w_0\in W$.  After choosing a \emph{realization} $\mathfrak{h}$ of $W$, there is an associated category $\SBim=\SBim(W,\mathfrak{h})$ of Soergel bimodules (or its diagrammatic version; see \cite{EW} and references therein).  Given a subset $I\subset S$, we let $\SBim_I\subset \SBim$ denote the full monoidal, idempotent complete, subcategory generated by Bott-Samelson bimodules $B_s$ with $s\in I$.  Note that $\SBim_I$ is just the category of Soergel bimodules associated to the \emph{parabolic subgroup} $W^I\subset W$, defined using the given realization $\mathfrak{h}$ of $W$.

Let $\FT:=F_{w_0}^{\otimes 2}$ denote the Rouquier complex for the ``full twist'' in $\KC^b(\SBim)$, and let $\FT_I:=F_{w_I}^{\otimes 2}$, where $w_I$ is the longest element of $W^I\subset W$.  Set $\FT_{S/I}:= \FT\otimes \FT_I\inv$.  Equivalently, $\FT_{S/I}=F_{v_I\inv}\otimes F_{v_I}$, where $v_I\in W$ denote a shortest length representative of the coset $w_0 W_I$.

\begin{conjecture}\label{conjecture:generalCox}
Let $\pi_L,\pi_R:\SBim\rightarrow \SBim_I$ denote the left and right adjoints to the fully faithful inclusion $\SBim_I\rightarrow \SBim$.  Then $\FT_{S/I}$ tensor commutes all with complexes in $\KC^b(\SBim_I)$ up to natural homotopy equivalence, and
\[
\pi_R(X)\simeq \pi_L(\FT_{S/I} \otimes X) \in \KC^b(\SBim_I),
\]
naturally in $X\in \KC^b(\SBim)$.
\end{conjecture}

The results in this paper prove this conjecture in the special case of subgroups $S_r\times (S_1)^{n-r}\subset S_n$.

\subsection{Khovanov-Rozansky homology}
\label{subsec:introKR}

Finally, we apply the above results to relate the ``top" and ``bottom" $a$-degrees in the Khovanov-Rozansky homology.
This categorifies a result of K\'alm\'an \cite{Kalman} relating the ``top" and "bottom" parts of the HOMFLY-PT polynomial.

\begin{theorem}
\label{thm: intro top bottom}
For any braid $\beta$ on $n$ strands one has
$$
\HHH^n(\beta\otimes \FT_n)(-2n)\simeq \HHH^0(\beta).
$$
\end{theorem}

\begin{remark}
For torus knots, Theorem \ref{thm: intro top bottom} was conjectured in \cite{G,ORS,GORS}. It was recently proved by the fourth author \cite{Nak} based on the explicit computation of the Khovanov-Rozansky homology for torus knots \cite{Mellit}.
\end{remark}

We also prove a ``folk result'' relating the Hochschild cohomology (or homology) of $X\in \KC^b(\SBim_n)$ and its dual $X^\vee$.

\begin{theorem}
Let $\widetilde{\HH}^k(M):=\HH^k(M)(-2k)$.  Then
\[
\widetilde{\HH}^k(X)\cong \widetilde{\HH}^{n-k}(X^\vee)^\vee
\]
as complexes of $R$-modules, for all $X\in \KC^b(\SBim_n)$.
\end{theorem}
If $\b$ is a braid, let $r(\b)$ denote the reversed braid, defined by $r(\sigma_i^{\pm})=\sigma_i^{\pm}$ for each elementary braid generator $\sigma_i^{\pm}$, and $r(\b \b') = r(\b') r(\b)$ for all $\b,\b'\in \Br_n$.  If $L$ is a link which is presented as the closure of a braid $\b$, then the mirror image $\overline{L}$ can be presented as the closure of the reversed inverse braid $r(\b\inv)$. There is an anti-involution of $\SBim_n$ defined by switching the right and left actions of all bimodules, which exchanges the Rouquier complexes for $\b$ and its reverse $r(\b)$.  It follows that $\HH^k(\b)\cong \HH^k(r(\b))$. Since the Rouquier complexes satisfy $F(\b\inv) = F(\b)^\vee$, we obtain the following corollary:

\begin{corollary}
We have $\widetilde{\HH}^k(\b)\cong \widetilde{\HH}^{n-k}(\b\inv)^\vee$ as complexes of $R$-modules, for all braids $\b\in \Br_n$.  In particular the complexes which compute the Khovanov-Rozansky homologies of $L$ and $\overline{L}$ are graded dual as complexes of free $R$-modules.
\end{corollary}

\subsection{Remark on conventions}
\label{subsec:intro other rings}
In this paper we have made the choice to work with honest Soergel bimodules rather than the diagrammatic version of Elias-Khovanov \cite{EK}, so that we may discuss Hochschild (co)homology.  Also, we have chosen to work over an infinite field $\k$ of characteristic $\neq 2$, so that Soergel's results apply.  When $\k$ is a more general ring, one can still define $\SBim_n$, but one loses control over the indecomposable objects in $\SBim_n$.  Nonetheless, we believe that all of our main results should hold over $\Z$, but where $\KC^b(\SBim_n)$ gets replaced by the homotopy category of complexes of Bott-Samelson bimodules.  These two categories are equivalent when Soergel's results apply.

\section*{Acknowledgements}

The authors would like to thank  Tam\'as K\'alm\'an, Andrei Negu\cb{t}, Alexei Oblomkov and Jacob Rasmussen for the useful discussions. We also thank American Institute of Mathematics, where a part of this work was done, for hospitality.
E. G.~ was partially supported by the NSF grants DMS-1700814, DMS-1760329, and the Russian Academic Excellence Project 5-100.  M.H.~ was supported by NSF grant DMS-1702274 and also partially supported by NSF grants DMS-1664240 and DMS-1255334. A.M.~ was supported by Austrian Science Fund (FWF) projects Y963-N35 and P-31705. K.N.~ was supported by JSPS KAKENHI Grant Number JP19J12350.

\section{Decategorified story}

\subsection{Jones-Ocneanu trace}

Let $\mathbb{H}_n$ be the Hecke algebra for $S_n$.  We adopt Soergel's conventions below.  We will work over the field $\Q(q)$ or occasionally $\Q(q)\subset \Q(v)$, where $q=v^{-2}$. (The variable $v$ corresponds to the grading downshift endofunctor (1) of $\SBim_n$.)
The algebra $\HM_n$ is formally generated by elements $H_1,\ldots,H_{n-1}$ modulo the braid relations and 
\[
(H_i+v)(H_i-v\inv)=0.
\] 
Given $w\in S_n$ with a reduced expression $w=s_{i_1}\cdots s_{i_k}$, we define $H_w=H_{i_1}\cdots H_{i_k}$. 
 
The algebra $\mathbb{H}_n$ has two standard bases as a $\Q(v)$-vector space, namely the positive standard basis $\{H_w\}_{w\in S_n}$ and the negative standard basis $\{{H_{w\inv}\inv}\}_{w\in S_n}$.

\begin{remark}
It is also common to express everything above in terms of $T_w:=(-v)^{\ell(w)}H_w$, where $\ell(w)$ is the Bruhat length of $w$.
\end{remark}
Jones and Ocneanu \cite{Jones,HOMFLY}  defined a trace function $\Tr \colon \mathbb{H}_n \rightarrow \Q(v)[a]$,
which (up to a normalization factor) agrees with the HOMFLY-PT polynomial. We define Jones-Ocneanu trace in section \ref{subsec:decat decomp} and list some of its most important properties here. 

For $x\in \mathbb{H}_n$, $\Tr(x)$ is a polynomial in $a$ of degree at most $n$ with coefficients being rational functions in $v$. Let $\Tr^n(x)$ (resp. $\Tr^0(x)$) be the coefficient of $a^n$ (resp. $a^0$) in $\Tr(x)$. 


\begin{lemma}
\label{lem: trace as coefficient at 1}
If we express $x\in \mathbb{H}_n$ in the positive and negative standard bases as
$$
x=\sum_{w\in S_n}\phi_w H_w=\sum_{w\in S_n}\psi_w H_{w\inv}\inv,
$$
then we have $\Tr^n(x)=(1-q)^{-n}\phi_1$ and $\Tr^0(x)=(1-q)^{-n}\psi_1$.
\end{lemma}

Indeed, it follows from Lemma \ref{lemma:ptr decat} below that for $w\neq 1$ one has $\Tr^n(H_w)=0$ and $\Tr^0(H_{w\inv}\inv)=0$, while $\Tr^0(1)=\Tr^n(1)=(1-q)^{-n}$.

Let $\e\colon \mathbb{H}_n \rightarrow \Q(v)$ be the vector space projection $\sum_{w\in S_n}\psi_w H_{w\inv}\inv\mapsto \psi_1$, and let
$(-)^\vee\colon \mathbb{H}_n \rightarrow \mathbb{H}_n$ be the ring anti-automorphism defined by $H_i^\vee=H_{i}\inv$ and $v^\vee=v^{-1}$.
(We remark that $(H_{w\inv}\inv)^{\vee}=H_{w^{-1}}$.)
Define a pairing $\ip{-,-}\colon \mathbb{H}_n\times \mathbb{H}_n\rightarrow \Z[q^{\pm}]$ over $\mathbb{H}_n$ by $\ip{x,y}:=\e(yx^\vee)$.
By the definition, we have $\ip{xz,y} = \ip{x,yz^{\vee}}$. We also have  $\e(xy)=\e(yx)$, hence $\ip{zx,y} = \ip{x,z^{\vee}y}$. 

\begin{remark}
This is the pairing which is categorified by the hom pairing of Soergel bimodules (see for instance \cite{EK}, modulo conventions).
\end{remark}

\begin{theorem}[\cite{Kalman}]
For all $x\in \mathbb{H}_n$ one has $\ip{x\FT,1}=\ip{1,x}^{\vee}$ and  $\Tr^n(x\FT)=\Tr^0(x)$.  
\end{theorem}

In \cite{Kalman}, this is proved by the fact that $\phi_{w_0}=\psi_{w_0}$ in the expansions of $x\HT$.
Equivalently, one has $\ip{\HT,H_{w\inv}\inv}=0$ if $w\neq w_0$ and $\ip{\HT,\HT^{-1}}=1$.
It is also known that the basis $\{H_{w\inv}\inv\}_{w\in S_n}$ is the dual basis of $\{H_w\}_{w\in S_n}$ with respect to this pairing.
(In fact, this orthogonality holds for Hecke algebras with any Coxeter group. For more details and a categorified result, see Appendix \ref{app:anytype}.)

\subsection{Partial traces}
\label{subsec:decat decomp}

We define Jones-Ocneanu trace on the Hecke algebra and its ``partial analogues" following \cite{Jones}.  The somewhat nonstandard conventions below are chosen to match with the categorical picture, in $\SBim_n$, discussed later.

The algebra $\mathbb{H}_n$ may be regarded as a bimodule over $\mathbb{H}_{n-1}$, and we have an isomorphism
\[
\mathbb{H}_n \  \  \cong  \ \  \mathbb{H}_{n-1} \ \oplus  \ \mathbb{H}_{n-1}\otimes_{\mathbb{H}_{n-2}}\mathbb{H}_{n-1}
\]
This isomorphism is not canonical, but two natural choices are (recall that $q=v^{-2}$)
\begin{equation}\label{eq:twoisos}
 \mathbb{H}_{n-1} \ \oplus  \ \mathbb{H}_{n-1}\otimes_{\mathbb{H}_{n-2}}\mathbb{H}_{n-1}\buildrel\Phi^{\pm} \over \longrightarrow \mathbb{H}_n, \qquad\qquad (x,y\otimes z)\mapsto (1-q)x + y H_{n-1}^\pm z.
\end{equation}

With respect to these isomorphisms, the induced projections $\pTr^\pm:\mathbb{H}_n\rightarrow \mathbb{H}_{n-1}$ are characterized by
\[
\pTr^\pm(x) = \frac{1}{1-q}x,\qquad\quad  \pTr^{\pm}(xH_{n-1}^{\pm} y) =0
\]
for all $x,y\in \mathbb{H}_{n-1}$.  We have a close relationship between $\pTr^\pm$ and the Jones-Ocneanu trace.

\begin{definition}
Let $\pTr:\mathbb{H}_n[a]\rightarrow \mathbb{H}_{n-1}[a]$ be the map defined by $\pTr(x)=\pTr^-(x) + a\pTr^+(x)$.
\end{definition}

\begin{lemma}\label{lemma:ptr decat}
The map $\pTr:\mathbb{H}_n\rightarrow \mathbb{H}_{n-1}$ satisfies 
\[
\pTr(x)=\frac{1+a}{1-q} x, \qquad \qquad \pTr(x H_{n-1} y) = -v xy,\qquad\qquad \pTr(x H_{n-1}\inv y) =  av\inv xy,
\]
for all $x,y \in \mathbb{H}_{n-1}$.  Furthermore, the composition
\[
\mathbb{H}_n[a] \buildrel \pTr \over\longrightarrow \mathbb{H}_{n-1}[a]\buildrel \pTr \over\longrightarrow \cdots \buildrel \pTr \over\longrightarrow   \mathbb{H}_{0}[a]=\Q(q)[a]
\]
is the Jones-Ocneanu trace.
\end{lemma}

In particular, we have $(\pTr^-)^n(x)=\Tr^0(x)$ and $(\pTr^+)^n(x)=\Tr^n(x)$ for $x \in \mathbb{H}_n$.
We wish to categorify this story.  It will turn out that the functors which categorify $\pTr^\pm$ are related by a relative version of Serre duality.

\section{Background}
\label{sec:background}



In this section we discuss the backround in Soergel bimodules and Khovanov-Rozansky homology.  Throughout this paper, let $\k$ be an infinite field of characteristic $\neq 2$.  This guarantees that the results of \cite{Soergel} apply, though we will only need these results in type $A$.

\subsection{Soergel bimodules}
\label{subsec:SBim}
Fix an integer $n\geq 1$, and let $R=\k[x_1,\ldots,x_n]$. The ring $R$ is graded such that the variables $x_i$ have degree 2.  The notions of an $R$-bimodule and an $R^e$-module will be identified, where $R^e = R\otimes_\k R = \k[x_1,\ldots,x_n,x_1',\ldots,x_n']$.  



Let $B_i$ denote the \emph{elementary bimodules} $B_1,\ldots,B_{n-1}$, defined by
$$
B_i=R\otimes_{R^{(i,i+1)}}R(1),
$$
where $R^{(i,i+1)}\subset R$ is the subalgebra of polynomials which are symmetric in $x_i,x_{i+1}$.  We may identify $B_i(-1)$ with $\k[x_1,\ldots,x_n,x'_1,\ldots, x'_n]$ modulo the ideal generated by
\[
x_i+x_{i+1}-x'_{i}-x'_{i+1},\qquad \ \ \ x_ix_{i+1}-x'_ix'_{i+1}, \qquad \ \ \ x_j-x'_j\ (j\neq i,i+1).
\]

\begin{definition}
Let $\SBim_n$ denote the full subcategory of graded $(R,R)$-bimodules generated by $R,B_1,\ldots,B_{n-1}$ and closed under direct sums, tensor product $\otimes_R$, gradings shifts $(\pm 1)$, and direct summands (i.e.~retracts).  A tensor product of shifts of elementary bimodules $B_i$ is called a \emph{Bott-Samelson bimodule}; by convention, $R$ (the ``empty tensor product'') is also regarded as a Bott-Samelson bimodule.
\end{definition}


\begin{notation}
Henceforth, the tensor product $\otimes_R$ will simply be denoted $\otimes$.
\end{notation}  

The category $\SBim_n$ is additive but not abelian.  An important result of Soergel \cite{Soergel} states that,  up to isomorphism and shift, the indecomposables $B_w$ in $\SBim_n(\k)$ are in one-to-one correspondence with $w\in S_n$.  Furthermore, if $s_{i_1}\cdots s_{i_\ell}$ is a reduced expression of $w\in S_n$, then the Bott-Samelson bimodule $B_{s_1}\otimes\cdots \otimes B_{s_\ell}$ has a unique summand isomorphic to $B_w$, and the remaining summands are $B_v$ for elements $v\in S_n$ of shorter length. 


The morphisms in $\SBim_n$ are degree preserving $R$-bilinear maps.  In this paper we almost exclusively work with space of morphisms of arbitrary degree $\Hom^{\Z}(M,N)=\bigoplus\Hom(M,N(i))$.  In fact, $\Hom^{\Z}$ will occur so often that we will simply write $\Hom^\Z=\Hom$ by abuse, and we will write $\Hom^0$ when we wish to emphasize degree zero morphisms.  By convention, every arrow $M\rightarrow N$ will be a degree preserving map in whatever category, unless otherwise specified.

Now, the (graded) hom spaces $\Hom(M,N)$ are graded $R$-bimodules, via
\[
f\cdot \phi\cdot g : m\mapsto f\phi(m)g = \phi(fmg)
\]
for all $f,g\in R$, $m\in M$, and $\phi\in \Hom(M,N)$.

Given two complexes $A=(A^\bullet,d_A)$ and $B=(B^\bullet,d_B)$ in $\KC^b(\SBim_n)$, we define a complex $\Homc(A,B)=(C_k,d_{\Homc})$ where  
$$
C^k=\prod_{i\in \Z}\Hom(A^i,B^{i+k}),\ \qquad  d_{\Homc}(f):=d_B\circ f - (-1)^k f\circ d_A.
$$
\begin{notation}
It would be more precise to write $\Homc^{\Z\times \Z}_{\KC^b(\SBim_n)}(A,B)$ instead of $\Homc(A,B)$.
\end{notation}

\subsection{Hochschild (co)homology}
\label{subsec:hochschild}
If $M$ is a graded $R$-bimodule, the zeroth Hochschild cohomology $\HH^0(M)$ is defined to be the sub-bimodule of $M$ spanned by homogeneous elements $m\in M$ with $x_im=mx_i$ for all $i=1,\ldots,n$.  Note that by definition,
\begin{equation}
\label{eq:HH^0 as hom}
\HH^0(M)\cong \Hom(R,M).
\end{equation}
Dually, the zeroth Hochschild homology $\HH_0(M)$ is defined to be the quotient bimodule $M/[R,M]$, that is to say $M$ modulo the $\k$-submodule spanned by commutators $x_im - mx_i$ for all homogeneous elements $m\in M$ and all $i=1,\ldots,n$.

The higher derived functors of $\HH^0$ and $\HH_0$ are denoted by $\HH^k$ and $\HH_k$; they are zero outside the range $0\leq k\leq n$.

Self-duality of the Koszul resolution of $R$ as a graded bimodule implies the following.
\begin{lemma}\label{lemma:H homology is cohomology}
For each graded $R$-bimodule $M$, we have
\[
\HH_{k}(M)\cong \HH^{n-k}(M)(-2n).
\]\qed
\end{lemma}

We will regard $\HH^k(M)$ and $\HH_k(M)$ as graded $R,R$-bimodules on which the left and right $R$-actions coincide.  That is to say, $\HH^k$ and $\HH_k$ may be viewed as endofunctors of the category of graded $R$-bimodules.

We have the following ``Markov moves'' for $\HH_k$ and $\HH^k$.
\begin{lemma}\label{lemma:HH markovs}
Let $M\in \SBim_{n-1}$ be given.  Then
\begin{subequations}
\begin{equation}\label{eq:trivialMarkovUpper}
\HH^k(M\sqcup \one_1)\cong \Big(\HH^k(M)\sqcup \one_1\Big) \oplus \Big(\HH^{k-1}(M)\sqcup \one_1\Big)(2)
\end{equation}
\begin{equation}\label{eq:interestingMarkovUpper}
\HH^k\Big((M\sqcup \one_1)\otimes B_{n-1}\otimes (N\sqcup \one_1)\Big)\cong  \Big(\HH^k(M\otimes N)\sqcup \one_1\Big)(-1) \oplus \Big(\HH^{k-1}(M\otimes N)\sqcup \one_1\Big)(3)
\end{equation}
\begin{equation}\label{eq:trivialMarkovLower}
\HH_k(M\sqcup\one_1)\cong \Big(\HH_k(M)\sqcup \one_1\Big) \oplus \Big(\HH_{k-1}(M)\sqcup \one_1\Big)(-2)
\end{equation}
\begin{equation}\label{eq:interestingMarkovLower}
\HH_k\Big((M\sqcup \one_1)\otimes B_{n-1}\otimes (N\sqcup \one_1)\Big)\cong  \Big(\HH_k(M\otimes N)\sqcup \one_1\Big)(1) \oplus \Big(\HH_{k-1}(M\otimes N)\sqcup \one_1\Big)(-3)
\end{equation}
\end{subequations}
\end{lemma}
Here $M\sqcup \one_1 = M[x_n]$ is the induced $R$-bimodule.
\begin{proof}
Standard, see \cite{Kh07} and also Proposition 3.10 in \cite{Hog18-GT}. Note that \eqref{eq:trivialMarkovLower} and \eqref{eq:interestingMarkovLower} follow from \eqref{eq:trivialMarkovUpper} and \eqref{eq:interestingMarkovUpper} using the isomorphism $\HH_k(M)\cong \HH^{n-k}(M)(-2n)$.
\end{proof}

\begin{corollary}\label{co:HHfree}
For each $1\leq k\leq n$ and each $B\in \SBim_n$, the Hochschild cohomology $\HH^k(B)$ is a free $R$-module of finite rank.
\end{corollary}
\begin{proof}
Since summands of free graded finite rank $R$-modules are free of finite rank, it suffices to prove in the case when $B=B_{i_1}\otimes\cdots\otimes B_{i_r}$ is a Bott-Samelson bimodule.  We have
\begin{equation}\label{eq: sts}
(B_i\otimes B_{i+1} \otimes B_i) \oplus B_{i+1} \cong (B_{i+1}\otimes B_{i}\otimes B_{i+1})\oplus B_{i},
\end{equation}
\begin{equation}\label{eq: ss}
B_i^{\otimes 2} \cong B_i(1)\oplus B_i(-1),
\end{equation}
so a straightforward induction allows us to reduce to the case when the index $n-1$ appears at most once among the indices $i_j$.  Applying \eqref{eq:trivialMarkovUpper} or \eqref{eq:interestingMarkovUpper} we reduce to the statement for $n-1$.
\end{proof}
Thus, we may view $\HH^k$ as an endofunctor of $\SBim_n$: the input is an arbitrary Soergel bimodule and the output is a direct sum of finitely many copies of $\one$ with shifts.

\subsection{Duals}
\label{subsec:duals}
The category $\SBim_n$ has a contravariant functor $(-)^\vee:\SBim_n\rightarrow \SBim_n$ so that $B^\vee$ is the two-sided dual (or biadjoint) to $B$.  This functor satisfies $B_i^\vee = B_i$ for all $i$ and $(M\otimes N)^\vee\cong N^\vee\otimes M^\vee$.  The duality functor comes from the observation that each bimodule $B_1,\ldots,B_{n-1}$ is a Frobenius algebra object in $\SBim_n$.  Precisely, there are canonical chain maps
\[
B_i\otimes B_i\rightarrow B_i(-1),\qquad B_i(-1)\rightarrow R,\qquad R\rightarrow B(1),\qquad B_i(1) \rightarrow B_i\otimes B_i.
\]
The first and and third of these maps give $B_i(-1)$ the structure of an algebra object, and the second and fourth maps give $B_i(1)$ the structure of a coalgebra object.  Moreover, the composition of the first two  defines a map $B_i \otimes B_i\rightarrow R$, the composition of the last two defines a map $R\rightarrow B_i\otimes B_i$, and these maps realize the fact that $B_i$ is self-dual.   In general we have natural isomorphisms:
\begin{subequations}
\begin{equation}\label{duality for SBim}
\Hom(A,B)\cong \Hom(R,B\otimes A^\vee) \cong \Hom(A\otimes B^\vee,R),
\end{equation}
\begin{equation}\label{duality for SBim right}
\Hom(A,B)\cong  \Hom(R,A^\vee \otimes B)\cong \Hom(B^\vee\otimes A,R) .
\end{equation}
\end{subequations}


\begin{remark}
\label{rem:duality with R action}
Recall that $\Hom(A,B)$ is a graded $R$-bimodule.  The isomorphisms \eqref{duality for SBim} are isomorphisms of graded left $R$-modules, while \eqref{duality for SBim right} are isomorphisms of graded right $R$-modules.  In fact we can say more; for instance the right action on $\Hom(A,B)$ can be understood as corresponding to the $R$-action on $\Hom(R,B\otimes A^\vee)$ via ``middle multiplication'' on $B\otimes A^\vee$. 
\end{remark}

\begin{remark}
We can also consider the duality isomorphisms for complexes.  For each $A,B\in \KC^b(\SBim_n)$ we have
$$
\Homc(A,B) \cong \Homc(\one,  B\otimes A^{\vee}) \cong \Homc(A\otimes B^{\vee},\one)
$$
as complexes of $R$-modules (with the left $R$-action on $\Hom(A,B)$), and
\[
\Homc(A,B) \cong \Homc(\one,A^\vee\otimes B)\cong \Homc(B^\vee\otimes A,\one)
\]
as complexes of $R$-modules (with the right $R$-action on $\Homc(A,B)$).
\end{remark}

It can be useful to phrase this categorical duality in terms of the usual duality in the category of $R$-modules.  If $M$ is a graded $R$-module, we let $M^\star:=\Hom^\Z_R(M,R)$ denote the graded $R$-module of homs.  There is a natural map $M\to (M^{\star})^{\star}$, which is an isomorphism if $M$ is free and finitely generated.

If $B$ is an $R,R$-bimodule, then we can forget the left action, obtaining a dual bimodule $B^\star:=\Hom_{\k\otimes R}(B,R)$, or we can forget the right $R$-action, obtaining a dual bimodule ${}^\star B:=\Hom_{R\otimes \k}(B,R)$
\begin{lemma}\label{lemma:twodualities}
We have natural isomorphisms
\[
B^\star \cong \ B^\vee \ \cong \ {}^\star B
\]
for $B\in \SBim_n$.
\end{lemma}
\begin{proof}
We will define inverse isomorphisms $\Phi:B^\star \leftrightarrow B^\vee\cong \Hom_{\Z\otimes R}(R, B^\vee) :\Psi$.  Let $f:B\rightarrow R$ be a morphism of graded right $R$-modules.  Define $\Phi(f)$ to be the composition
\[
R \rightarrow B\otimes_R B^\vee \rightarrow R\otimes_R B^{\vee} \cong B^{\vee},
\]
where the first map is given by duality and the second is $f\otimes \Id$.  

In the other direction, let $g:R\rightarrow B^\vee$ be a morphism of graded right $R$-modules, and define $\Psi$ to be the composition
\[
B\cong R\otimes_R B \rightarrow B^\vee\otimes_R B \rightarrow R,
\]
where the second map is $g\otimes \Id_{B^\vee}$ and the last map is given by duality.  It is an easy exercise to show that $\Psi$ and $\Phi$ are inverse isomorphisms of graded bimodules $B^\star \cong \Hom_{\Z\otimes R}(R, B^\vee)$; they are clearly natural in $B$.

The proof that $B^\vee\cong {}^\star B$ naturally is similar.
\end{proof}

\begin{notation}
Henceforth, if $M$ is a finitely generated free $R$-module, then $M$ will be regarded as an object of $\SBim_n$, and $M^\star$ will be denoted by $M^\vee$.
\end{notation}


The following is standard.
\begin{lemma}\label{lemma:free homs}
For all $M,N\in \SBim_n$ the hom bimodule $\Hom(M,N)$ is free as a left or right $R$-module.
\end{lemma}
\begin{proof}
By Remark \ref{rem:duality with R action}, it suffices to prove the lemma in the special case $M=R$; in this case $\Hom(R,N)=\HH^0(N)$ is free by Corollary \ref{co:HHfree}.
\end{proof}

Now we consider how the duality functor interacts with Hochschild (co)homology.

\begin{lemma}
\label{lem: HH mirror}
For each $B\in \SBim_n$ one has $\HH^k(B^\vee)(-2n)=\HH^{n-k}(B)^{\vee}$. 
\end{lemma}


\begin{proof}



It is easy to see from the definition that complexes of $R$-modules computing Hochschild cohomology of $B$ the Hochschild homology of $B^{\vee}$ are dual to each other.  Since the cohomology of both complexes are free over $R$, the cohomologies are dual to each other as well, i.e.~$\HH_k(B^\vee)\cong \HH^k(B)^\vee$. The Lemma now follows from Lemma \ref{lemma:H homology is cohomology}.
\end{proof}

The symmetries between $\HH^k$, $\HH^{n-k}$, $\HH_k$, and $\HH_{n-k}$ are quite a bit more attractive (and easy to remember) after a change in normalization.

\begin{definition}\label{def:normalizedAgrading}
Let $\widetilde{\HH}^k(M):=\HH^k(M)(-2k)$ and $\widetilde{\HH}_{k}(M):=\HH_k(M)(2k)$.
\end{definition}

\begin{proposition}\label{prop:HHsymmetries}
If $k+l=n$, then we have
\[
\widetilde{\HH}^k(M)\cong \widetilde{\HH}_{l}(M) \cong  \widetilde{\HH}^{l}(M^\vee)^\vee.
\]
for all $M\in \SBim_n$.  These are isomorphisms of functors from $\SBim_n$ to the category of finitely generated free graded $R$-modules.
\end{proposition}
\begin{proof}
This is just a restatement of Lemma \ref{lemma:H homology is cohomology} and Lemma \ref{lem: HH mirror}.
\end{proof}

\begin{remark}
When expressing the Poincar\'e series of Khovanov-Rozansky homology, the variable $a=AQ^{-2}$ is often used instead of $A$. This precisely corresponds to replacing $\HH^k$ by $\widetilde{\HH}^k$. (Here, $Q$ denotes the degree in Soergel bimodules and $A$ denotes the usual Hochschild degree.)
\end{remark}

\begin{corollary}\label{cor:HH_0 as hom}
We have $\HH_0(M)\cong \Hom(M,R)^\vee$, natural for $M\in \SBim_n$.
\end{corollary}
\begin{proof}
Indeed
\[
\Hom(M,R)^\vee\cong \Hom(R,M^\vee)^\vee \cong \HH^0(M^\vee)^\vee\cong \HH_0(M).
\]
Each of these isomorphisms is functorial in $M$.  
\end{proof}

\subsection{Rouquier complexes}
\label{subsec:rouquier}

Let $\KC^b(\SBim_n)$ denote the homotopy category of bounded complexes of Soergel bimodules, with differentials of degree $+1$. 
It inherits all the structures from $\SBim_n$: it is an additive tensor category with (two-sided) duals and shifts $(\pm 1)$.
It is also triangulated, with cohomological shift functors $[\pm 1]$.

We will need special complexes
$$
\rouq_i \ \ := \ \ 0\rightarrow \underline{B_i}\rightarrow {R}(1)\rightarrow 0,
$$
$$
\rouq_i^{-1}\ \ :=\ \ 0 \rightarrow R(-1)\rightarrow \underline{B_i}\rightarrow 0,
$$
where we have underlined terms in degree zero.  Note that $\rouq_i=\Cone(b_i)[-1]$ and $\rouq_i\inv = \Cone(b_i^*)$ for some distinguished morphisms $b:B_i\to R(1)$ and $b^*:R(-1)\to B_i$ which are adjoint to each other. 

\begin{theorem}[\cite{Rouquier}]
The complexes $\rouq_i$ and $\rouq_i^{-1}$ satisfy the braid relations up to canonical homotopy equivalence:
\begin{enumerate}
\item $\rouq_i\otimes \rouq_{i+1}\otimes \rouq_i\ \simeq\  \rouq_{i+1}\otimes \rouq_i\otimes \rouq_{i+1}$,
\item $ \rouq_{i}\otimes \rouq_{i}^{-1} \ \simeq \  \one\ \simeq \ \rouq_i^{-1}\otimes \rouq_i$,
\item $\rouq_i\otimes \rouq_j \ \simeq \ \rouq_j\otimes \rouq_i$ if $|i-j|\ge 2$.
\end{enumerate}
\end{theorem}

As a corollary, for any braid $\beta$ expressed as a product of the standard generators $\sigma_i^\pm$ we can define a complex $F(\beta)$ as a tensor product of complexes of the form $\rouq_i^{\pm 1}$; the resulting complex depends only on $\beta$ and not the expression as a product of generators, up to coherent homotopy equivalence.  Since $\rouq_i^\vee\simeq \rouq_i^{-1}$, we have
$$
F(\beta)^{\vee}\simeq F(\beta^{-1})\simeq F(\beta)^{-1}
$$
for all $\beta$.

\begin{definition}\label{def:hull}
Given a collection of complexes $X_1,\ldots,X_r\in\KC^b(\SBim_n)$, define the \emph{graded triangulated hull} $\ip{X_1,\ldots,X_r}$  to be the smallest full triangulated subcategory of $\KC^b(\SBim_n)$ containing $X_1,\ldots,X_r$ and closed under grading shifts.   We also refer to $\ip{X_1,\ldots,X_r}$ as the \emph{span} of $\{X_i\}$.
\end{definition}

\begin{proposition}\label{prop:rouquier generation}
For each $w\in S_n$ let $F_w$ and $F_w\inv$ denote the Rouquier complexes associated to the positive braid lift of a chosen reduced expression for $w$.  Then $\{F_w\}_{w\in S_n}$ and $\{F_w\inv\}_{w\in S_n}$ span $\KC^b(\SBim_n)$.
\end{proposition}


\subsection{Triply graded Khovanov-Rozansky homology}
\label{subsec:KR homology}
Every additive functor can be extended  term-wise to complexes. In particular, for a complex $C=(C^\bullet,d_C)$ and $j\in \Z$ we can define Hochschild cohomology $\HH^{j}(C)$ as the complex
\[
\cdots\rightarrow \HH^j(C^k) \rightarrow \HH^j(C^{k+1})\rightarrow \cdots
\]
whose differential is just the functor $\HH^j$ applied to $d_C$.  If $C\simeq D$ then $\HH^j(C)\simeq \HH^j(D)$ as complexes of $R$-modules.  Each of the natural isomorphisms in \S \ref{subsec:hochschild} extends without trouble to the corresponding categories of complexes.  In particular if $X\in \KC^b(\SBim_n)$ is a complex of Soergel bimodules then
\[
\HH^k(X)(-2n)\cong \HH^{n-k}(X^\vee)^\vee \cong \HH_{n-k}(X)
\]
naturally.

\begin{definition}
For each $X\in \KC^b(\SBim_n)$, let $\HH^\bullet(X):=\bigoplus_{k=0}^n \HH^k(X)$ and $\HH_\bullet(X):=\bigoplus_{k=0}^n \HH_k(X)$.
\end{definition}
The homology of $\HH^\bullet(X)$ is often denoted $\HHH^\bullet(X)$, but we will not need this. 

\begin{lemma}[Markov moves]\label{lemma:Markov2}
Let $X,Y\in \KC^b(\SBim_{n-1})$ be arbitrary.  Then
\[
\HH^k(X\otimes \rouq_{n-1}\otimes Y)\simeq \HH^k(X\otimes Y)[-1](1)
\]
and
\[
\HH^k(X\otimes \rouq_{n-1}\inv\otimes Y)\simeq \HH^{k-1}(X\otimes Y)(-1).
\]
\end{lemma}
\begin{proof}
Standard; see \cite{Kh07} and Proposition 3.10 in \cite{Hog18-GT}.
\end{proof}

\begin{corollary}
\label{cor:negative stabilization}
In the notation of Lemma \ref{lemma:Markov2} we have
\[
\HH_0(X\otimes \rouq_{n-1} \otimes Y)\simeq 0,\qquad\qquad  \HH^0(X\otimes \rouq_{n-1}\inv\otimes Y)\simeq 0.
\]
\end{corollary}

We also record an important connection of $\HH$ to the Jones-Ocneanu trace:

\begin{theorem}[\cite{Kh07}]
The Euler characteristic of $\HH$ equals the Jones-Ocneanu trace:
\[
\sum_i a^{i} \chi(\HH^i(X)(-2i))=\Tr([X])
\]
for all $X\in \KC^b(\SBim_n)$, where $\chi$ is a graded Euler characteristic and 
$[X]$ is the class of $X$ in the Grothendieck group $K_0(\SBim_n)\cong\mathbb{H}_n$.
\end{theorem}

\begin{corollary}
We have $\chi(\HH^0(X))=\Tr^0([X])$ and $\chi(\HH^n(X))=\Tr^n([X])$. 
\end{corollary}

\section{Decompositions of categories}

We first recall some background on categorical idempotents and semi-orthogonal decompositions.

\subsection{Adjoints to inclusions}
\label{subsec:adjointsgeneral}
Let $\AS$ be a category and $\BS$ a category with functor $\sigma:\BS\rightarrow \AS$.  The following is classical; its proof is an exercise in category theory and Yoneda embedding.  See also Theorem 1 in \S IV.3 of \cite{Maclane}.

\begin{lemma}\label{lemma:adjunctions yield idempotents}
Suppose $\sigma:\BS\rightarrow \AS$ has a left adjoint $\pi_L:\AS\rightarrow \BS$.  Then $\sigma$ is fully faithful if and only if the counit of the adjunction is an isomorphism $\pi_L\circ \sigma\rightarrow \Id_{\BS}$.

Dually, if $\sigma$ has a right adjoint $\pi_R:\AS\rightarrow \BS$, then $\sigma$ is fully faithful if and only if the unit of the adjunction is an isomorphism $\Id_\BS\rightarrow \pi_R\circ \sigma$.\qed
\end{lemma}
If these equivalent conditions hold, then $\EB_L:=\sigma\circ \pi_L$ and $\EB_R:=\sigma\circ \pi_R$ are idempotent functors $\AS\rightarrow \AS$ with essential image equivalent to $\BS$.   We discuss these next, after an example.

\begin{example}
Let $A$ be an algebra and $I\subset A$ a two-sided ideal.  Let $\AS$ be the category of $A$-modules and $\BS\subset \AS$ the category of objects on which $I$ acts by zero.  Then the inclusion $\BS\rightarrow \AS$ has a left adjoint sending an $A$-module $M$ to $M/IM$, and a right adjoint sending $M$ to the annihilator of $I$ in $M$.  Both functors can be regarded as idempotent endofunctors of $\AS$ with image $\BS$.
\end{example}

\begin{definition}\label{def:localization}
A \emph{localization functor} on a category $\CS$ is a pair $(\EB,\eta)$ consisting of an endofunctor $\EB:\CS\rightarrow \CS$ and a natural transformation $\eta:\Id_{\CS}\rightarrow \EB$ such that $\EB\eta$ and $\eta \EB$ are isomorphisms $\EB\buildrel \cong \over \rightarrow \EB\EB$. 

A \emph{colocalization functor} on $\CS$ is a pair $(\EB,\e)$ consisting of an endofunctor $\EB:\CS\rightarrow \CS$ and a natural transformation $\e:\EB\rightarrow \Id_{\CS}$ such that $\EB\e$ and $\e \EB$ are isomorphisms $\EB\EB \buildrel\cong\over \rightarrow \EB$.
\end{definition}

\begin{lemma}\label{lemma:idempotents yield adjunctions}
Let $\EB$ be an endofunctor of a category $\CS$ with essential image $\BS\subset \CS$.  If $\eta:\Id_\CS\rightarrow \EB$ 
(resp. $\e:\EB\rightarrow \Id_{\CS}$)
gives $\EB$ the structure of a localication (resp.~colocalization) endofunctor, then $\EB$ defines a left (resp.~right) adjoint to the inclusion $\BS\rightarrow \CS$.

Conversely if $\BS\subset \CS$ is a full subcategory such that the inclusion $\sigma:\BS\rightarrow \CS$ admits a left (resp.~right) adjont $\pi:\CS\rightarrow \BS$, then the $\sigma \pi$ is a localization (resp.~colocalization) endofunctor.
\end{lemma}
\begin{proof}
Suppose that $\EB:\CS\rightarrow \CS$ is a localization functor.  We must show that if $\EB(Y)\cong Y$ and $X\in \AS$ is arbitrary, then the unit map $\eta_X:X\rightarrow \EB(X)$ induces an isomorphism
\[
\Hom_{\AS}(X,Y)\cong \Hom_{\BS}(\EB(X),Y)=\Hom_{\AS}(\EB(X),Y).
\]
This is proven, for example in \cite{Krause10}, Proposition 2.4.1.

The statement about colocalization functors follows by taking opposite categories.

Conversely, if $\sigma:\BS\rightarrow \CS$ admits a left adjoint $\pi:\CS\rightarrow \BS$, then the counit of the adjunction $\e:\pi\sigma\rightarrow \Id_{\BS}$ is an isomorphism of functors.  Let $\eta:\Id_\CS\rightarrow \sigma\pi$ be the unit of the adjunction.  Then $\eta \sigma\pi, \sigma\pi\eta: \sigma\pi \rightarrow \sigma\pi \sigma\pi$ are isomorphisms with inverse $\sigma\e \pi$.
\end{proof}




%

\begin{remark}
In linear algebra there can be many idempotent endomorphisms of a vector space $V$ which project onto a given subspace $W\subset V$ (the embedding  $W\rightarrow V$ has many left inverses).  In the realm of category theory, idempotents are quite a bit more rigid.  Indeed, the inclusion $\BS\rightarrow \AS$ of a full subcategory can be the image of at most one localization functor and at most one colocalization functor (because left and right adjoints to a given functor are unique when they exist). 

\end{remark}

\newcommand{\ES}{\EuScript{E}}
\newcommand{\IB}{\mathbf{I}}
\subsection{Semi-orthogonal decompositions}
\label{subsec:semiortho}
In the setting of triangulated categories it is possible to discuss the notion of complementary idempotent functors.  Let $\MS$ a triangulated category and let $\AS$ be a triangulated monoidal category which acts on $\MS$ by exact endofunctors (see \cite{Hog17b} for details).  

\begin{example}
The only example we will need in this paper is the following.  Let $\CS$ be an additive $\k$-linear category with $\End(\CS)$ the category of linear endofunctors.  Then the homotopy category of complexes $\MS:=\KC^b(\CS)$ is an $\AS$-module category where $\AS:=\KC^b(\End(\CS))$.  In other words, complexes of endofunctors of $\CS$ can be thought of as endofunctors of $\KC^b(\CS)$.  For instance if $F:\CS\rightarrow \CS$, then $F$ can be thought of as a complex in $\KC^b(\End(\CS))$ of degree zero, and the action of $F$ on complexes is the usual:
\[
F(X)  \  := \ \cdots \buildrel F(d) \over \longrightarrow F(X^k) \buildrel F(d) \over \longrightarrow F(X^{k+1}) \buildrel F(d) \over \longrightarrow\cdots 
\]
\end{example}

\begin{definition}\label{def:complementary idempotents}
An \emph{idempotent triangle} in $\AS$ is a distinguished triangle of the form 
\begin{equation}\label{eq:idemp triangle}
\PB\buildrel \e \over \rightarrow \IB\buildrel \eta \over \rightarrow \QB\buildrel\d\over \rightarrow \PB[1]
\end{equation}
such that $\PB\QB \simeq 0 \simeq \QB\PB$ in $\AS$.  In such a triangle, we refer to $\PB$ and $\QB$ as \emph{complementary idempotents} in $\AS$.

A \emph{unital idempotent} in $\AS$ is an object $\QB\in \AS$ with a map $\eta:\IB\rightarrow \QB$ such that $\eta\QB$ and $\QB\eta$ are isomorphisms $\QB\rightarrow \QB\QB$.  A \emph{counital idempotent} in $\AS$ is an object $\PB\in \AS$ with a map $\e:\PB\rightarrow \IB$ such that $\PB\e$ and $\e\PB$ are isomorphisms $\PB\PB\rightarrow \PB$.
\end{definition}

We will denote the identity of $\AS$ by $\IB$, and we will write the monoidal structure in $\AS$ simply by juxtaposition of functors.  Moreover, if $f$ is a morphism in $\AS$ and $\XB$ is an object of $\AS$, then we write $f \XB$ and $\XB f$ for $f\otimes \Id_{\XB}$ and $\Id_{\XB} \otimes f$, respectively.  This notation is compatible with the usual conventions for writing horizontal compositions of functors and natural transformations.

\begin{remark}
If $\EB\in \AS$ has the structure of a (co)unital idempotent, then its action on $\MS$ is a (co)localization functor.
\end{remark}

Observe that if \eqref{eq:idemp triangle} is an idempotent triangle in $\AS$ and $X\in \MS$ is an object of the $\AS$-module category $\MS$, then $X$ fits into a distinguished triangle
\[
\PB(X)\rightarrow X\rightarrow \QB(X)\rightarrow \PB(X)[1].
\]
If $\QB(X)\simeq 0$ then the first map is an isomorphism in $\MS$ $\PB(X)\rightarrow X$, while if $\PB(X)\simeq 0$ then the second map is an isomorphism $X\rightarrow \QB(X)$, by properties of distinguished triangles.  Thus, we arrive at the following:

\begin{observation}
The essential image of $\PB$ acting on $\MS$ coincides with the kernel of $\QB$ acting on $\MS$, and vice versa.
\end{observation}

\begin{lemma}\label{lemma:idemp triangles}
In any idempotent triangle \eqref{eq:idemp triangle} the maps $\e$ and $\eta$ give $\PB$ and $\QB$ the structure of a counital and unital idempotent, respectively.  Conversely:
\begin{enumerate}
\item if $(\QB,\eta)$ is a unital idempotent in $\AS$, then $\PB:=\Cone(\eta)[-1]$ has the structure of a counital idempotent.
\item If $(\PB,\e)$ is a counital idempotent in $\AS$ then $\QB:=\Cone(\e)$ has the structure of a unital idempotent in $\AS$.
\end{enumerate}
\end{lemma}
\begin{proof}
Straightforward.
\end{proof}

\begin{definition}\label{def:complement}
If $(\PB,\e)$ is a counital idempotent, then $\PB^c:=\Cone(\e)$ is called the \emph{(unital) complement} of $\PB$.  If $(\QB,\eta)$ is a unital idempotent, then $\QB^c:=\Cone(\eta)[-1]$ is called the \emph{(counital) complement} of $\QB$.
\end{definition}
Note that $(\PB^c)^c\simeq \PB$ and $(\QB^c)^c\simeq \QB$.

Next we discuss the relation between categorical idempotents and semi-orthogonal decompositions. The following property is key.
\begin{lemma}\label{lemma:idempt semi ortho}
Let $(\PB,\e)$ be a counital idempotent in $\AS$ with complement $\QB:=\PB^c$.  Then for each $X,Y\in \MS$ we have
\[
\Hom_{\MS}(\PB(X),\QB(Y))\cong 0.
\]
\end{lemma}
\begin{proof}
See Theorem 4.13 in \cite{Hog17b}.
\end{proof}

This semi-orthogonality property is closely tied with the discussion of adjoints to inclusions in the previous section.  Let $\QB(\MS)$ denote the essential image of $\QB$, and let $Y\in \QB(\MS)$ be given, so that $\QB(Y)\simeq Y$.  Let $X$ be arbitrary.  Then Lemma \ref{lemma:idempt semi ortho} implies that $\Hom_\MS(\PB(X),Y)\cong 0$. Applying $\Hom_{\MS}(-,\QB(Y))$ to the distinguished triangle $\PB(X)\rightarrow X\rightarrow \QB(X)\rightarrow \PB(X)[1]$, we find that precomposing with $\eta_X:X\rightarrow \QB(X)$ is an isomorphism
\[
\eta_X^\ast : \Hom_{\QB(\MS)}(\QB(X),Y)\buildrel \cong \over \rightarrow \Hom_\MS(X,Y)
\]
This shows that $\QB$, when regarded as a functor $\MS\rightarrow \QB(\MS)$ is the left adjoint to the inclusion $\QB(\MS)\rightarrow \MS$.

Similarly, if $X\in \PB(\MS)$ and $Y\in \MS$ is arbitrary then applying $\Hom_{\MS}(X,-)$ to the distinguished triangle $\PB(Y)\rightarrow Y\rightarrow \QB(Y)\rightarrow \PB(Y)[1]$ and using Lemma \ref{lemma:idempt semi ortho} shows that post-composing with $\e_Y:\PB(Y)\rightarrow Y$ is an isomorphism
\[
(\e_Y)_\ast : \Hom_{\PB(\MS)}(X,\PB(Y))\buildrel\cong\over \rightarrow \Hom_{\MS}(X,Y).
\]
Thus, $\PB$ defines the right adjoint to the inclusion $\PB(\MS)\rightarrow \MS$.

\begin{definition}\label{def:semiorthog}
Let $\NS_1,\NS_2\subset \MS$ be full triangulated subcategories.  We say that $(\NS_1,\NS_2)$ is a \emph{semi-orthogonal decomposition} of $\MS$, written $\MS\simeq (\NS_1\rightarrow \NS_2)$, if
\begin{enumerate}
\item each object $X\in \MS$ fits into a distinguished triangle
\begin{equation}\label{eq:semiortho triangle}
Y_2\buildrel \a\over \rightarrow X\buildrel \b\over \rightarrow Y_1\buildrel \gamma\over \rightarrow Y_2[1],\qquad\qquad Y_i\in \NS_i.
\end{equation}
\item $\Hom_{\AS}(Y_2,Y_1)\cong 0$ for all $Y_i\in \NS_i$.
\end{enumerate}
\end{definition}

The following is an immediate corollary of Lemma \ref{lemma:idempt semi ortho} and the fact that the kernel of a (co)unital idempotent in $\AS$ is the image of its complement.
\begin{corollary}\label{cor:semiortho from idempts}
If $\PB$ (resp.~$\QB$) is a counital (resp.~unital) idempotent in $\AS$ then each $\AS$-module category $\MS$ inherits a semi-orthogonal decomposition $\MS\simeq (\operatorname{ker} \PB \rightarrow \operatorname{im} \PB)$ (resp.~$\MS\simeq (\operatorname{im} \QB \rightarrow \operatorname{ker}\QB)$).\qed
\end{corollary}

\begin{definition}\label{def:perp}
If $\NS\subset \MS$ is a full triangulated subcategory then $\NS^\perp\subset \MS$ and ${}^\perp\NS\subset \MS$ denote the full subcategories of objects $X$ such that $\Hom_\MS(B,X)\cong 0$ (resp.~$\Hom_\MS(X,B)\cong 0$) for all $B\in \NS$.
\end{definition}

\begin{theorem}\label{thm:semiorthogonality}
If $\MS\simeq (\NS_1\rightarrow \NS_2)$ is a semi-orthogonal decomposition then
\begin{enumerate}
\item Each distinguished triangle \eqref{eq:semiortho triangle} is unique up to unique isomorphism (extending the identity of $X$).
\item The categories $\NS_1$ and $\NS_2$ determine each other: $\NS_1={}^\perp\NS_2$ and $\NS_2=\NS_1^\perp$.
\end{enumerate}
\end{theorem}
It is also possible to show that in any semi-orthogonal decomposition $\MS\simeq (\NS_1\rightarrow \NS_2)$ the ``projections onto $\NS_i$'' are well-defined idempotent endofunctors of $\MS$ (the projection onto $\NS_1$ is unital, while projection onto $\NS_2$ is counital).  However, we will not need this.

\begin{proof}
Let $\MS\simeq (\NS_1\rightarrow \NS_2)$ be a semi-orthogonal decomposition.  We will first show (1).  Let
\[
Z_2\buildrel \a'\over \rightarrow X \buildrel \b' \over \rightarrow Z_1 \buildrel \gamma'\over \rightarrow Z_2[1]
\]
be a distinguished triangle with $Z_i\in \NS_i$.  We must construct a map to this triangle from \eqref{eq:semiortho triangle} and show that this map is unique.  First, since $\Hom(Y_2,Z_1)\cong 0$ it follows that:
\begin{itemize}
\item post-composing with the map $\a':Z_2\rightarrow X$ gives an isomorphism $\Hom(Y_2,Z_2)\rightarrow \Hom(Y_2,X)$.
\item pre-composing with the map $\b:X\rightarrow Y_1$ gives an isomorphism $\Hom(Y_1,Z_1)\rightarrow \Hom(X,Z_1)$.
\end{itemize}
One sees this by applying $\Hom(Y_2,-)$ to the distinguished triangle $Z_2\to X\to Z_1\to $ and by applying $\Hom(-,Z_1)$ to the distinguished triangle $Y_2\to X\to Y_1\to$.  Thus there is a unique $f_2:Y_2\rightarrow Z_2$ such that $\a'\circ f_2 = \a$ and a uniqe $f_1:Y_1\rightarrow Z_1$ such that $f_1\circ \b = \b'$.  Thus, there is a unique map of triangles as claimed.  The uniqueness also implies that all such maps are isomorphisms of triangles.

Now we show (2).  The containment $\NS_2\subset \NS_1^\perp$ holds by the assumption of semiorthogonality.  To prove the opposite containment $\NS_1^\perp\subset \NS_2$, suppose $Z_2\in \NS_1^\perp$ is given.  We have a distinguished triangle
\[
Y_2\rightarrow Z_2\rightarrow Y_1\rightarrow Y_2[1]
\]
with $Y_i\in \NS_i$.   We want to show that $Y_2=Z_2$.  To do this it suffices to show that $Y_1=0$.  Apply the cohomological functor $\Hom(-,Y_1)$, and consider the resulting long exact sequence in cohomology:
\[
\Hom_\MS(Y_1,Y_1[k])\rightarrow  \Hom_{\MS}(Z_2,Y_1[k])\rightarrow \Hom_{\MS}(Y_2,Y_1[k])\rightarrow \Hom_{\MS}(Y_1,Y_1[k+1]).
\]
The second term is zero for all $k$ by hypothesis that $Z_2\in \NS_1^\perp$, and the third term is zero for all $k$ by semi-orthogonality between $\NS_1$ and $\NS_2$.  It follows that $\End_{\MS}(Y_1)\cong 0$, hence $Y_1=0$, hence $Z_2=Y_2\in \NS_2$, as claimed.

A similar argument shows that $\NS_1={}^\perp \NS_2$.
\end{proof}

The following is useful in determining the images and kernels of (co)unital idempotents.

\begin{lemma}\label{lemma:kernels and images}
Let $\{X_i\}$ and $\{Y_j\}$ be a collection of objects in $\MS$.  Let $\EB$ be a unital or counital idempotent in $\AS$, with complement $\EB^c$.  Assume that $\EB(X_i)\simeq 0$ for all $i$ and $\EB(Y_i)\simeq Y_j$ for all $j$.  Then:
\begin{enumerate}
\item $\EB(X)\simeq 0$ for all $X\in \ip{X_i}_i$.
\item $\EB(Y)\simeq Y$ for all $Y\in \ip{Y_j}_j$.
\item If $\{X_i\}_i$ and $\{Y_j\}_j$ together span all of $\CS$, then $\{X_i\}_i$ spans the kernel of $\EB$ while $\{Y_j\}_j$ spans the image of $\EB$.
\end{enumerate}
\end{lemma}
\begin{proof}
Statements (1) follows by the five lemma, and (2) is equivalent to (1) for the complement $\EB^c$.  Now, for (3), let us assume for the sake of concreteness that $\EB$ is a counital idempotent with counit $\e:\EB\rightarrow \IB$.  Now, let $Z\in \MS$ be given.  We may express $Z$ as an iterated cone involving various $X_i$ and $Y_j$ (with shifts).  But $\EB$ annihilates each $X_i$, which implies that $\EB(Z)$ can be expressed as some iterated cone involving $\EB(Y_j)\simeq Y_j$, since $\EB$ is assumed to be exact.  It follows that the essential image of $\EB$ is spanned by the $Y_j$.  The same argument applied to $\EB^c$ shows that the kernel of $\EB$ is spanned by the $X_i$.
\end{proof}

\section{Semi-orthogonal decompositions of the Hecke category}
\label{sec:semiortho}

\subsection{Two adjoints} 
\label{subsec:twoadjoints}

Let $\CS_\one\subset \SBim_n$ denote the full subcategory generated by the trivial bimodule $\one=R$ and its shifts.   In this section we make the key observation that the inclusion $\iota:\CS_\one\rightarrow \SBim_n$ has both a left adjoint and a right adjoint, given by $\HH_0$ and $\HH^0$.

First, for each $B\in \SBim_n$ we will identify $\HH^0(B)$ with the sub-bimodule of $B$ consisting of elements $b\in B$ with $fb=bf$ for all $f\in R$.  This is free as a graded $R$-module, hence is isomorphic to a finite direct sum of shifted copies of $R$. Thus, $\HH^0$ can be thought of as a functor from $\SBim_n\rightarrow \CS_\one$, or as an endofunctor of $\SBim_n$.  The inclusion $\HH^0(B)\subset B$ defines a natural transformation $\e:\HH^0\rightarrow \Id_{\SBim_n}$.

Also, for each $B$ recall that $\HH_0(B)$ is the quotient of $B$ by the sub-bimodule spanned by elements of the form $fb - bf$ for all $b\in B$ and all $f\in R$.  This, too, is free as a graded $R$-module, hence can be regarded as an object of $\CS_\one$.  The projection $B\rightarrow \HH_0(B)$ defines a natural transformation $\eta:\Id_{\SBim_n}\rightarrow \HH_0$.

\begin{lemma}
We have the following:
\begin{enumerate}
\item the map $\e:\HH^0\rightarrow \Id_{\SBim_n}$ makes $\HH^0$ into a counital idempotent endofunctor of $\SBim_n$ with image $\CS_\one$. 
\item the map $\eta:\Id_{\SBim_n}\rightarrow \HH_0$ makes $\HH_0$ into a unital idempotent endofunctor of $\SBim_n$ with image $\CS_\one$.
\end{enumerate}
Furthermore, when regarded as functors $\SBim_n\rightarrow \CS_\one$, $\HH^0$ and $\HH_0$ are right and left adjoints to the inclusion $\CS_\one\rightarrow \SBim_n$, respectively.
\end{lemma}
\begin{proof}
It is clear that if $B$ is a direct sum of shifted copies of $R$, then $\HH^0(B) = B$ and $\e_B:\HH^0(B)\rightarrow B$ is the identity map.  On the other hand, $\HH^0(B)\in \CS_\one$ for all $B\in \SBim_n$, so the image of $\HH^0$ is $\CS_\one$.  A moment's thought confirms that $\e_{\HH^0(B)}$ and  $\HH^0(\e_{B})$ are natural isomorphisms $\HH^0(\HH^0(B))\buildrel\cong\over\rightarrow \HH^0(B)$.  This proves (1).  The proof of (2) is similar.

The fact that $\HH^0$ and $\HH_0$ define  right and left adjoints to the inclusion of $\CS_\one\rightarrow \SBim_n$ now follows from general properties of (co)unital idempotent endofunctors. 
\end{proof}




\begin{example}
In case $n=2$, $\HH^0$ sends $R\mapsto R$ and $B_s\mapsto R(1)$, while $\HH_0$ sends $R\mapsto R$ and $B_s\mapsto R(-1)$.
\end{example}

\subsection{Triangulated perspective on $\HH^0$ and $\HH_0$}
\label{subsec:absolute semiortho decomp}

We already know that the \emph{images} of the idempotent functors $\HH^0$, $\HH_0$ are both $\CS_\one$.   What about the kernels?  To answer this extend the functors $\HH^0$ and $\HH_0$ to complexes.  For each $\XB\in \KC^b(\SBim_n)$, we have:
\begin{enumerate}
\item $\HH^0(X)\subset X$ is the subcomplex consisting of those homogeneous elements $b$ with $fb=bf$ for all $f\in R$.
\item $\HH_0(X)$ is the quotient of $X$ by the subcomplex spanned by elements $fb-bf$ for all homogeneous $b\in X$, $f\in R$.
\end{enumerate}

Following Definition \ref{def:complement}, we define endofunctors $Q^{\pm}:\KC^b(\SBim_n)$ which are complementary to $\HH_0,\HH^0$, as follows. For each $X\in \KC^b(\SBim_n)$ define
\[
Q^-(X):=\Cone(\HH^0(X)\rightarrow X),
\]
\[
Q^+(X):=\Cone(B\rightarrow \HH_0(X))[-1].
\]
In order to use the setup of \S \ref{subsec:semiortho}, one should set $\MS:=\KC^b(\SBim_n)$ and $\AS:=\KC^b(\End(\SBim_n))$.  Then $\HH_0$ is a unital idempotent in $\End(\SBim_n)\subset \AS$, and $Q^+$ is its complement.  Similarly $\HH^0$ is a counital idempotent in $\AS$, and $Q^-$ is its complement.

\begin{example}
In case $n=2$, recall that $\HH_0$ sends $R\mapsto R$ and $B_s\mapsto R(1)$.  The unit of $\HH_0$ consists of the identity map $R\rightarrow R$ and the dot map $B_s\rightarrow R(1)$.  Thus, the complementary idempotent $Q^+$ sends $R$ to the contractible complex $\Cone(\Id_R)$, and sends $B_s$ to $\rouq_s$, the Rouquier complex associated to $s$.

A similar computation shows that $Q^-$ sends $R$ to a contractible complex and $B_s$ to $\rouq_s\inv$.
\end{example}





Next we intend to use Lemma \ref{lemma:kernels and images} to describe the images and kernels of $\HH^0,\HH_0$.

\begin{lemma}
\label{lem: negative and pos tw}
For all $w\neq 1$ one has $\HH^0(\rouq_w\inv)\simeq 0$ and $\HH_0(\rouq_w)\simeq 0$.  
\end{lemma}
\begin{proof}
If $w\neq 1$, we can find an integer $1\leq k\leq n-1$ and a minimal length representative of $w$ which contains a unique copy of $s_{k}$ and no $s_j$ for $j>k$, so that $\rouq_w^{\pm 1}=A\otimes \rouq_k^{\pm 1}\otimes B$ for $A, B\in \KC^b(\SBim_{k})$. Then Corollary \ref{cor:negative stabilization} gives the desired relations.
\end{proof}

\begin{definition}\label{def:TScats}
Let $\TS^{+}:=\ip{\rouq_w}_{w\neq 1}$ and  $\TS^{-}:=\ip{\rouq_w\inv}_{w\neq 1}$ be the full triangulated subcategories  of $\KC^b(\SBim)$ generated by $\rouq_w$ (resp.~$\rouq_w\inv$) and their shifts, where $w$ ranges over all non-identity elements of $S_n$.
\end{definition}

\begin{theorem}
\label{th: semiortho 1}
The kernels of $\HH^0$ and $\HH_0$, regarded as endofunctors $\KC^b(\SBim_n)\rightarrow \KC^b(\SBim_n)$ are $\TS^{-}$ and $\TS^+$ respectively.  That is to say, we have two semi-orthogonal decompositions
\begin{eqnarray*}
\KC^b(\SBim_n) & \simeq & (\TS^-\rightarrow \ip{\one})\\
& \simeq & (\ip{\one}\rightarrow \TS^+) .
\end{eqnarray*}
in which the projections onto $\ip{\one}$ are $\HH^0$ and $\HH_0$, respectively.
\end{theorem}



\begin{proof}
The triangulated category $\KC^b(\SBim_n)$  is spanned by $\{\rouq_w\inv\}_{w\in S_n}$.  Now, since $\HH^0$ annihilates $\rouq_w\inv$ for $w\neq 1$ (Lemma \ref{lem: negative and pos tw}) and fixes $\rouq_1=\one$, it follows from Lemma \ref{lemma:kernels and images} that $\ip{\one}$ and $\TS^-$ are the image and kernel of $\HH^0$, respectively.  

A similar argument shows that $\ip{\one}$ and $\TS^+$ are the image and kernel of $\HH_0$.  This yields the semi-orthogonal decompositions in the statement  and completes the proof.



\end{proof}

\begin{remark}
Theorem \ref{th: semiortho 1} can be regarded as a categorification of Lemma \ref{lem: trace as coefficient at 1}.
Indeed, $\HH^0$ is a counital idempotent functor  whose image is $\ip{\one}$ and whose kernel is $\ip{\rouq_w\inv}_{w\neq 1}$. This is a categorification of the fact that the coefficient of $a^0$ in the Jones-Ocneanu trace picks out the coefficient of 1 in the $\{\rouq_w\inv\}$-expansion of $x\in \mathbb{H}_n$.

On the other hand, $\HH_0$ is a unital dempotent endofunctor with image $\ip{\one}$ and kernel $\ip{\rouq_w}_{w\neq 1}$.  
Since $\HH_0(X)\cong \HH^n(X)(-2n)$, this yields a categorification of the fact that the coefficient of $a^n$ in Jones-Ocneanu trace picks out the coefficient of 1 in the $\{\rouq_w\}$-expansion of $x\in \mathbb{H}_n$.
\end{remark}



\subsection{Partial trace functors}
\label{subsec:partial traces}
We introduce some intermediate categories between $\SBim_n$ and $\CS_\one$.  For each parabolic subgroup $P:=S_{m_1}\times\cdots\times S_{m_r}\subset S_n$, let $\SBim_{m_1,\ldots,m_r} = \SBim(P)$ denote the full subcategory of $\SBim_n$ spanned by the bimodules $B_w$ with $w\in P$ and their shifts. Note that $\SBim_{1,\ldots,1} =\CS_{\one}$.
 We regard $\SBim_{m_1,\ldots,m_r}$ as a full monoidal subcategory of $\SBim_n$, and so $\SBim_n$ can be thought of as a \emph{bimodule category} over $\SBim_{m_1,\ldots,m_r}$.

\begin{definition}\label{def:ptr}
For $M\in \SBim_n$ define $\pi^-(M)$ and $\pi^+(M)$ to be the kernel and cokernel of $x_n\otimes 1 - 1\otimes x_n$ acting on $M$, respectively.
\end{definition}

The following is obvious.

\begin{lemma}\label{lemma:tr is bilinear}
The functors $\pi^\pm$ are $\SBim_{n-1,1}$-bilinear, in the sense that
\[
\pi^{\pm}(X\otimes Y\otimes Z) \cong X\otimes \pi^{\pm}(Y)\otimes Z
\]
for all $X,Z\in \SBim_{n-1,1}$ and all $Y\in \SBim_n$; this isomorphism is natural in $X,Y,Z$.\qed
\end{lemma}

\begin{lemma}\label{lemma:image of ptr}
If $M\in \SBim_n$, then $\pi^{\pm}(M)$ is an object of $\SBim_{n-1,1}$, so $\pi^{\pm}$ are well-defined functors $\SBim_n\rightarrow \SBim_{n-1,1}$.
\end{lemma}
\begin{proof}
An easy exercise shows that $\pi^{\pm}(B_{n-1})\cong R(\pm 1)$ and $\pi^{\pm}(R)=R$, hence
\begin{equation}
\label{eq:partial trace of B 1}
\pi^\pm(X) = X
\end{equation}
and
\begin{equation}\label{eq:partial trace of B 2}
\pi^{\pm} (X\otimes B_{n-1} \otimes Y)\cong X\otimes Y(\pm 1)
\end{equation}
for all  $X,Y\in \SBim_{n-1,1}$.  Every indecomposable Soergel bimodule appears as a direct summand of some Bott-Samelson bimodule in which $B_{n-1}$ appears at most once, so it follows that $\pi^\pm(Z)\in \SBim_{n-1,1}$ for all $Z\in \SBim_n$.
\end{proof}

Since $\pi^+(M)$ is a sub-bimodule of $M$, the inclusion $\pi^+(M)\hookrightarrow M$ defines a natural transformation $\e:\pi^+R\rightarrow \Id$.  Similarly, the projection $M\twoheadrightarrow \pi^-(M)$ defines a natural transformation $\eta:\Id\rightarrow \pi^-$.

\begin{proposition}\label{prop:partial trace as adjoints}
The functors $\pi^+$ and $\pi^-$ are the left and right adjoints to the inclusion $\SBim_{n-1,1}\hookrightarrow \SBim_n$.
\end{proposition}
\begin{proof}
It is clear that $\pi^+$ and $\pi^-$ are strict idempotent functors: $(\pi^{\pm})^2=\pi^\pm$.  In fact the inclusion $\pi^-(M)\rightarrow M$ makes $\pi^-$ into a localization functor and the projection $M\rightarrow \pi^+(M)$ makes $\pi^+$ into a colocalization functor.

Lemma \ref{lemma:idempotents yield adjunctions} then tells us that $\pi^-$ is right adjoint to the inclusion of its essential image $\SBim_{n-1,1}$, and similarly $\pi^+$ is left adjoint to the inclusion of its essential image, which again is $\SBim_{n-1,1}$.
\end{proof}

Since $\pi^+,\pi^-$ are additive functors $\SBim_n\rightarrow \SBim_{n-1,1}$ they can be extended to complexes.  Next we record how these functors interact with Rouquier complexes.  

\begin{lemma}\label{lemma:ptrRouquier}
We have the following ``Markov moves'' for $\pi^\pm$:
\begin{subequations}
\begin{equation}\label{eq:ptrMarkovZero}
\pi^+(X\otimes \rouq_{n-1} \otimes Y)\  \simeq  \ 0 \ \simeq \  \pi^-(X\otimes \rouq_{n-1}\inv \otimes Y)
\end{equation}
\begin{equation}\label{eq:ptrMarkovNonzero1}
\pi^+(X\otimes \rouq_{n-1}\inv \otimes Y) \cong \Cone\left(X\otimes Y(-1)\buildrel f\over\rightarrow X\otimes Y(1)\right),
\end{equation}
\begin{equation}\label{eq:ptrMarkovNonzero2}
\pi^-(X\otimes \rouq_{n-1} \otimes Y) \cong \Cone\left(X\otimes Y(-1)\buildrel f\over\rightarrow X\otimes Y(1)\right)[1]
\end{equation}
\end{subequations}
for every $X,Y\in \KC^b(\SBim_{n-1,1})$, where $f$ is ``middle multiplication'' by $x_{n-1}-x_n$.\qed
\end{lemma}

\begin{remark}
If we forget the action of $x_n$, then we obtain the more recognizable Markov moves $\pi^+(X\otimes \rouq_{n-1}\inv \otimes Y)\simeq X\otimes Y(-1)[1]$ and $\pi^-(X\otimes \rouq_{n-1}\otimes Y)\simeq X\otimes Y(1)[-1]$.  
\end{remark}
\begin{remark}
We refer to $\pTr^\pm$ as the \emph{partial trace functors}.  They can be regarded as functors $\SBim_{m,1,\ldots,1}\rightarrow \SBim_{m-1,1,1,\ldots,1}$ for all $1\leq m\leq n$.  Composing them yields
\[
\HH^0 = (\pi^-)^n : \SBim_n\rightarrow \SBim_{1,\ldots,1} = \CS_{\one} 
\]
and
\[
\HH_0 = (\pi^+)^n : \SBim_n\rightarrow \SBim_{1,\ldots,1} = \CS_{\one} 
\]
\end{remark}

\begin{remark}
Note that the inclusion $\SBim_{n-1}\rightarrow \SBim_{n-1,1}$ is not full: hom spaces in $\SBim_{n-1,1}$ are obtained from those in $\SBim_{n-1}$ by applying $-\otimes_\k \k[x_n]$. On the other hand, \emph{there is} a fully faithful functor on the level of homotopy categories $\sigma:\KC^b(\SBim_{n-1})\rightarrow \KC^b(\SBim_{n-1,1})$, sending
\[
X\mapsto \Cone\left( (X\sqcup \one_1) \buildrel x_n\over \longrightarrow (X\sqcup \one_1)\right).
\]
The Grothendieck groups of $\SBim_{n-1}$ and $\SBim_{n-1,1}$ are both naturally identified with $\mathbb{H}_{n-1}$, and $\sigma$ categorifies multiplication by $(1-q)$.

There is a forgetful functor $F:\SBim_{n-1,1}\rightarrow R_{n-1}\text{-bimod}$, where $R_{n-1}:=\k[x_1,\ldots,x_{n-1}]$.  Note that $F(B_w\sqcup \one_1)=B_w[x_n]$ for all $B_w\in \SBim_{n-1}$,  hence $F$ could be regarded (loosely speaking) as a categorification of multiplication by $\frac{1}{1-q}$.  In any case $F(\sigma(X))\simeq X$ for all $X\in \KC^b(\SBim_{n-1})$.

This explains how the apparently mysterious factor of $(1-q)$ in \eqref{eq:twoisos} is built-in to the categorical picture.
\end{remark}

\subsection{Relative semi-orthogonal decompositions}
\label{subsec:relative semiortho}
In the previous section we constructed idempotent endofunctors $\pi^{\pm}$ if $\SBim_n$ which project onto $\SBim_{n-1,1}$.  Now we would like to understand the semi-orthogonal decompositions they determine.  That is to say, we want to understand the kernels of these functors after extending to complexes $\pi^\pm:\KC^b(\SBim_n)\rightarrow \KC^b(\SBim_n)$.  To study the kernels, we define the complementary idempotents by the usual formulas (Definition \ref{def:complement}).  
\[
Q^-(X):=\Cone(\pi^-(X)\rightarrow X),\qquad\qquad Q^+(X):=\Cone(X\rightarrow \pi^+(X))[-1].
\]

\begin{definition}\label{def:relativeTpm}
Let $\US^\pm\subset \KC^b(\SBim_n)$ denote the full triangulated subcategory spanned by the Rouquier complexes $F_w^\pm$ with $w\in S_n \setminus (S_{n-1}\times S_1)$.  
Equivalently, $\US^\pm$ is the span of complexes of the form $X\otimes \rouq_{n-1}^\pm\otimes Y$ with $X,Y\in \KC^b(\SBim_{n-1,1})$.
\end{definition} 

\begin{theorem}\label{thm:relative semiortho}
The kernel of the idempotent endofunctor $\pi^{\pm}:\KC^b(\SBim_n)\rightarrow \KC^b(\SBim_n)$ is $\US^{\pm}$, so that we have semi-orthogonal decompositions
\begin{eqnarray*}
\KC^b(\SBim_n)
& \simeq & \left(\KC^b(\SBim_{n-1,1})\rightarrow \US^+\right)\\
& \simeq & \left(\US^-\rightarrow \KC^b(\SBim_{n-1,1}) \right).
\end{eqnarray*}
\end{theorem}
With respect to these semi-orthogonal decompositions $\pi^\pm$ can be described as the projection onto $\KC^b(\SBim_{n-1,1})$ with kernel $\US^{\pm}$. 
%

\begin{proof}
The proof is similar to the proof of Theorem \ref{th: semiortho 1}.  If $w\in S_n\setminus (S_{n-1}\times S_1)$, then $w$ can be presented as
\[
w = w's_{n-1} w''
\]
with $w',w''\in S_{n-1}\times S_1$.  For such $w$ we have
\[
\pi^+(F_w) \cong F_{w'}\otimes \pi^+(F_{n-1})\otimes F_{w''}\simeq 0
\]
and
\[
\pi^-(F_w\inv) \cong F_{w'}\inv\otimes \pi^+(F_{n-1}\inv)\otimes F_{w''}\inv\simeq 0.
\]

On the other hand, if $w\in S_{n-1}\times S_1$, then we have $F_w\in \KC^b(\SBim_{n-1,1})$, and so $\pi^\pm(F_w)\cong F_w$.  Thus, Lemma \ref{lemma:kernels and images} tells us that the image and kernel of $\pi^{\pm}:\KC^b(\SBim_n)\rightarrow \KC^b(\SBim_n)$ are $\KC^b(\SBim_{n-1,1})$ and $\US^\pm$, respectively.  This gives the desired semi-orthogonal decompositions and completes the proof.


\end{proof}


\section{Serre duality}
\label{sec:serre}
Let $\HT_n =\rouq_{w_0}$ denote the Rouquier complex associated to the half twist, and let $\FT_n:=\HT_n^{\otimes 2}$ denote the full twist.  When the index $n$ is understood, it will be omitted.

The purpose of this section is to prove the following.

\begin{theorem}
\label{thm:left right}
We have $\HH^0(\FT\inv\otimes X)\simeq \HH_0(X)$ and $\HH^0(X)\simeq \HH_0(\FT\otimes X)$ as complexes of $R$-modules, for all $X\in \KC^-(\SBim_n)$.
\end{theorem}

We will prove this theorem as a corollary of a certain ``relative version.'' 

\subsection{Jucys-Murphy braids and the splitting map}
\label{subsec:splitting map}
Define braids $\CL_n\in \Br_n$ inductively by $\CL_1=\one_1$ and
\[
\CL_n = \sigma_{n-1}(\CL_{n-1}\sqcup \one_1)\sigma_{n-1},\qquad \qquad n\geq 2.
\]
We will denote the braid $\CL_n$ and the Rouquier complex $F(\CL_n)$ by the same notation.  Note that
\[
\FT_n = \CL_2\otimes \CL_3\otimes\cdots \otimes \CL_n.
\]
For any $i$ we define the  chain map $\psi_i:F_i\to F^{-1}_i$ by the following diagram:
\[
\begin{tikzcd}
F_i \arrow{d}{\psi_i} & =  & {[\underline{B_i}} \arrow{r} \arrow{dr} & {R(1)]}\\
F^{-1}_i & = &  {[R (-1)}   \arrow{r} & {B_i]}\\
\end{tikzcd}
\]
Clearly, $\Cone[\psi_i]=[R(1)\to R(-1)]$. 
By combining these maps we get a ``splitting map'' $\Psi:\CL_n\rightarrow \one_n$. This map was studied in \cite[Section 4]{GH}
in a more complicated ``$y$-ified'' version.

\begin{lemma}\label{lemma:maps from Ln}
The space of chain maps $\Homc(\CL_n,\one_n)$ is homotopy equivalent to $R$, generated by the ``splitting map'' $\Psi:\CL_n\rightarrow \one_n$.  This splitting map becomes a homotopy equivalence after applying $\pi^+$.
\end{lemma}
\begin{proof}
The first statement follows from the second.  Indeed, if $\Psi$ becomes a homotopy equivalence after applying $\pi^+$ then we have
\[
\Homc(\CL_n,\one)\cong \Homc(\pi^+(\CL_n),\one) \simeq \Homc(\one,\one) = R,
\]
generated by $\Psi$, as claimed.   The first isomorphism above holds since $\pi^+$ is the left adjoint to the inclusion $\SBim_{n-1,1}\rightarrow \SBim_{n}$ (Proposition \ref{prop:partial trace as adjoints}).

To prove the second statement it suffices to show that the cone of the splitting map $\Psi:\CL_n\rightarrow \one_n$ is mapped to a contractible complex by $\pi^+$.  However, $\Cone(\Psi)$ is in the triangulated hull of the Rouquier complexes
\[
F(\sigma_{n-1}\cdots\sigma_{k+1}\sigma_{k}\sigma_{k+1}\cdots \sigma_{n-1}) \ \ \simeq \ \ F(\sigma_k\cdots\sigma_{n-2}\sigma_{n-1}\sigma_{n-2}\cdots\sigma_k),\qquad 1\leq k\leq n-1,
\]
Each of these is annihilated by $\pi^+$ by Lemma \ref{lemma:ptrRouquier}, and the lemma follows.
\end{proof}

\subsection{Relative Serre duality}
\label{subsec:relative serre}

\begin{lemma}\label{lemma:Ln action}
Tensoring on the left (or right) with $\CL_n$ restricts to an equivalence of categories $\US^-\rightarrow \US^+$ with inverse given by tensoring on the left (or right) with $\CL_n\inv$, where $\US^\pm$ is as in Definition \ref{def:relativeTpm}.
%
\end{lemma}
\begin{proof}
 It is clear that tensoring on the left with $\CL_n$ restricts to a functor $\US^-\rightarrow \US^+$, since $\CL_n \otimes F_{n-1}\inv \simeq F_{n-1}\otimes \CL_{n-1}$ and $\CL_n$ tensor commutes with all Rouquier complexes $F_w$ with $w\in S_{n-1}\times S_1$.  Similarly, tensoring on the left with $\CL_n\inv$ restricts to a functor $\US^+\rightarrow \US^-$.  These functors are inverse equivalences.  A similar argument takes care of tensoring on the right with $\CL_n^\pm$.
\end{proof}

\begin{theorem}\label{thm:relativeSerre}
We have $\pi^-(X)\simeq \pi^+(\CL_n\otimes X)\simeq \pi^+(X\otimes \CL_n)$ for all $X\in \KC^b(\SBim_n)$.  These homotopy equivalences are natural isomorphisms of functors $\KC^b(\SBim_n)\rightarrow \KC^b(\SBim_{n-1,1})$.
\end{theorem}
\begin{proof}
We only consider the equivalence $\pi^-(X)\simeq \pi^+(\CL_n\otimes X)$; the equivalence $\pi^-(X)\simeq \pi^+(X\otimes \CL_n)$ is proven similarly.  Let $X\in \KC^b(\SBim_n)$ be given.  We may as well assume that $X$ is expressed as
\[
X\simeq (\QB^-(X)  \buildrel[1]\over\rightarrow \pi^-(X))
\]
where $\QB^-(X)\in \US^-$.  Here, the label $[1]$ above an arrow indicates a chain map of degree 1 so for instance $B\simeq (C \buildrel[1]\over\rightarrow A)$ means that $B\simeq \Cone(C[-1]\rightarrow A)$ or, equivalently $B$ fits into a distinguished triangle of the form $A\rightarrow B\rightarrow C\rightarrow A[1]$.  Tensoring with $\CL_n$ yields 
\[
\CL_n\otimes X \ \simeq \ ( \CL_n\otimes \QB^-(X) \buildrel[1]\over\rightarrow \CL_n\otimes \pi^-(X) ).
\]
Since $\CL_n\otimes \QB^-(X)\in \US^+$, it follows that  $\pi^+(\CL_n\otimes \QB^-(X))\simeq 0$, hence
\[
\pi^+(\CL_n\otimes X)\simeq \pi^+(\CL_n\otimes \pi^-(X))\simeq  \pi^-(X).
\]
In this last equivalence we used Lemma \ref{lemma:maps from Ln}. 

Now we consider the naturality of this homotopy equivalence.  Let $X\rightarrow Y$ be a chain map.  Then in terms of the decompositions $X\simeq (\pi^-(X)\rightarrow \QB^-(X))$ and $Y\simeq (\pi^-(Y)\rightarrow \QB^-(Y))$, the chain map $f$ can be written as

\[
\begin{tikzpicture}[baseline=-2.8em]
\tikzstyle{every node}=[font=\small]
\node (a) at (-1,0) {$X$};
\node  at (0,0) {$\simeq$};
\node at (1.2,0) {$\Big($};
\node at (5.7,0) {$\Big)$};
\node (b) at (2,0) {$\QB^-(X)$};
\node (c) at (5,0) {$\pi^-(X)$};
\node (d) at (-1,-2.5) {$Y$};
\node  at (0,-2.5) {$\simeq$};
\node at (1.2,-2.5) {$\Big($};
\node at (5.7,-2.5) {$\Big)$};
\node (e) at (2,-2.5) {$\QB^-(Y)$};
\node (f) at (5,-2.5) {$\pi^-(Y)$};
\path[->,>=stealth',shorten >=1pt,auto,node distance=1.8cm,
  thick]
(b) edge node {$[1]$} (c)
(e) edge node {$[1]$} (f)
(a) edge node {$f$} (d)
(b) edge node {} (e)
(b) edge node {} (f)
(c) edge node {$\pi^-(f)$} (f);
\end{tikzpicture}.
\]
Applying $\pi^+(\CL_n\otimes -)$ and contracting the contractible terms $\pi^+(\CL_n\otimes \QB^-(X))$ and $\pi^+(\CL_n\otimes \QB^-(Y))$, we obtain a diagram which commutes up to homotopy:
\[
\begin{tikzpicture}[baseline=-2.8em]
\tikzstyle{every node}=[font=\small]
\node (a) at (-3,0) {$\pi^+(\CL_n\otimes X)$};
\node  at (-.7,0) {$\simeq$};
\node (b) at (2,0) {$\pi^+(\CL_n\otimes \pi^-(X))$};
\node (c) at (7,0) {$\pi^-(X)$};
\node (d) at (-3,-2.5) {$\pi^+(\CL_n\otimes Y)$};
\node  at (-.7,-2.5) {$\simeq$};
\node (e) at (2,-2.5) {$\pi^+(\CL_n\otimes \pi^-(Y))$};
\node at (4.5,0) {$\simeq$};
\node at (4.5,-2.5) {$\simeq$};
\node (f) at (7,-2.5) {$\pi^-(Y)$};
\path[->,>=stealth',shorten >=1pt,auto,node distance=1.8cm,
  thick]
(a) edge node {$\pi^+(\CL_n\otimes f)$} (d)
(b) edge node {$\pi^+(\CL_n\otimes \pi^-(f))$} (e)
(c) edge node {$\pi^-(f)$} (f);
\end{tikzpicture}.
\]
The second square commutes up to homotopy because it is induced by the splitting map $\CL_n\rightarrow \one$ (together with the observation that $\pi^+(\pi^-(X))=\pi^-(X)$ naturally).  This proves the statement about naturality.
\end{proof}







\subsection{Top versus bottom}
\label{subsec:top v bottom}
For each $1\leq r\leq n$, let $\US_r^\pm\subset \KC^b(\SBim_n)$ denote the full triangulated subcategories spanned by $X \otimes \rouq_{r-1}^\pm \otimes Y$ for all $X,Y\in \KC^b(\SBim(S_r\times S_1^{n-r})$.

Iterating the semi-orthogonal decompositions from \S \ref{subsec:relative semiortho} yields more sophisticated semi-orthogonal decompositions of $\KC^b(\SBim_n)$ of the form:
\[
\KC^b(\SBim_n) \ \simeq \ \left(\SBim(S_r\times S_1^{n-r}) \rightarrow \US_r^+\rightarrow \US_{r+1}^+\rightarrow\cdots \rightarrow \US_n^+\right),
\]
\[
\KC^b(\SBim_n) \ \simeq \ \left(\US_n^-\rightarrow \cdots\rightarrow \US_{r+1}^-\rightarrow \US_r^-\rightarrow \SBim(S_r\times S_1^{n-r})\right),
\]
or more generally
\[
\KC^b(\SBim_n) \ \simeq \ \left(\US_{l_b}^-\rightarrow \cdots\rightarrow \US_{l_1}^-\rightarrow \SBim(S_r\times S_1^{n-r}) \rightarrow \US_{k_1}^+\rightarrow \cdots \rightarrow \US_{k_a}^+\right),
\]
for any decomposition $\{r+1,\ldots,n\} = \{k_1<\cdots < k_a\}\sqcup \{l_1<\cdots<l_b\}$. 

The case $r=1$ yields the semi-orthogonal decompositions considered in \S \ref{subsec:absolute semiortho decomp}:
\[
\KC^b(\SBim_n) \ \simeq \ \left(\SBim(S_1^n)\rightarrow \TS^+\right) \ \simeq \left(\TS^-\rightarrow \SBim(S_1^n)\right),
\]
where $\TS^\pm\subset \KC^b(\SBim_n)$ is the full triangulated category spanned by the Rouquier complexes $\rouq_w^{\pm 1}$ with $w\neq 1$.

\begin{proof}[Proof of Theorem \ref{thm:left right}]
We must show that
\[
\HH^0(X)\simeq \HH_0(\FT_n\otimes X)
\]
for all $X\in \KC^b(\SBim_n)$, and that these homotopy equivalences yield a natural isomorphism of functors $\KC^b(\SBim_n)\rightarrow \KC^b(R-\text{gmod})$.

We will prove this by induction on $n$.  The base case $n=1$ is trivial.  Note that $\FT_n = \CL_n\otimes \FT_{n-1}$, $\HH^0(X) =(\pi^-)^n(X)$, and $\HH_0(X) =(\pi^+)^n(X)$.  Thus,
\begin{eqnarray*}
\HH^0(X)
&=& (\pi^-)^n(X) \\
&\simeq & \HH_0(\FT_{n-1} \otimes\pi^-(X))\\
& \simeq & \HH_0(\FT_{n-1}\otimes \pi^+(\CL_n \otimes X))\\
& \cong & \HH_0(\pi^+(\FT_{n-1}\otimes\CL_n \otimes X))\\
& =&  \HH_0(\FT_n\otimes  X).
\end{eqnarray*}
In the second line we used the induction hypothesis, in the third line we used Theorem \ref{thm:relativeSerre}, in the fourth we used Lemma \ref{lemma:tr is bilinear}, and the last line is clear.
\end{proof}

\begin{remark}
In light of the isomorphism $\HH_0(Y)\cong \HH^n(Y)(-2n)$, we prefer to view the result of Theorem \ref{thm:left right} as saying
\[
\HH^n(X)(-2n)\simeq \HH^0(\FT\inv\otimes X).
\]
\end{remark}

\begin{theorem}
\label{thm:serre}
The Rouquier complex for the full twist braid is a Serre functor for $\KC^b(\SBim_n)$. In other words,
for all $A,B\in \KC^b(\SBim_n)$ we have
$$
\Homc(A,B)=\Homc(B\otimes \FT,A)^{\vee},
$$
where the dual on the right hand side is defined using the left $R$-action on the hom complex.
\end{theorem}

\begin{proof}
Recall that $\SBim_n$ has duals, hence $\Homc(A,B)\cong\Homc(A\otimes B^\vee,\one)$ as complexes of left $R$-modules.
By Theorem \ref{thm:left right}, we have
$$
\Homc(A,B)\cong \Homc(A\otimes B^\vee,\one) \cong \HH_0(A\otimes B^\vee)^\vee 
$$
as complexes of left $R$-modules (the right action of $R$ on $\Hom(A,B)$ corresponds to ``middle multiplication'' on $A\otimes B^\vee$.  By Theorem \ref{thm:left right}, this latter complex is homotopy equivalent to
\[
\HH^0(A\otimes B^\vee\otimes \FT\inv)^\vee =\Homc(\one,A\otimes B^\vee\otimes \FT\inv)^\vee\cong \Homc(\FT\otimes B, A)^\vee
\]
where in the last complex we use the left $R$-action on $ \Homc(\FT\otimes B, A)$ when forming the dual $R$-module $(-)^\vee$.  These are homotopy equivalences of complexes of left $R$-modules. \end{proof}

\subsection{Soergel modules}
\label{sec:smod}

In this section we review the Serre duality for the category of {\em Soergel modules} $\SMod$, which is closely related to the Bernstein-Gelfand-Gelfand category $\mathcal{O}$ for the Lie algebra $\mathfrak{gl}_n$. Soergel modules are obtained from Soergel bimodules as quotients by the right $R$-action. Given $M\in \SBim_n$, we define
$$
\overline{M}=M\otimes_{R} R/(x_1,\ldots,x_n)=M\otimes_{R}\k.
$$
For example, $\overline{R}=\k$. Note that the $R\otimes_\k R$ action on the Soergel bimodule $M$ factors through the quotient $R\otimes_{R^{S_n}} R$, so the residual $R\otimes_\k \k$-action on $\overline{M}$ factors through $R\otimes_{R^{S_n}}\k$.  This latter ring is called the \emph{coinvariant ring}, denoted $C$; it is the quotient of $R$ by the ideal generated by positive degree symmetric functions in the $x_i$.  By definition, a morphism in $\SMod$ is a homogeneous $C$-linear map.

\begin{remark}
Let us explain the connection between $\SMod$ and the BGG categories $\OC$.  Let $\OC_0=\OC_0(\mathfrak{gl}_n)$ denote the principal block of the category $\OC$ for $\gl_n$ (i.e.~the block containing the trivial 1-dimensional representation).  This category has a special projective module $P=P_{w_0}$, the \emph{anti-dominant projective}, whose endomorphism ring is isomorphic to the coinvariant ring $C$.  Soergel \cite{Soergel} proved that the functor $\OC_0\rightarrow C\text{-mod}$ sending $M\mapsto \Hom_{\OC_0}(P,M)$ is fully faithful on projectives, hence identifies $D^b(\OC_0)$ ($\simeq$ the homotopy category of projectives) with a full subcategory of $\KC^b(C\text{-mod})$.  This full subcategory is precisely $\KC^b(\SMod)$.  An important consequence (and the original motivation) for such a description is that it yields a $\Z$-graded lift of $\OC_0$.
\end{remark}

Soergel modules do not form a monoidal category, but they form a module category over $\SBim_n$: given 
$A,B\in \SBim_n$, we have 
\begin{equation}
\label{eq: bimodule act on module}
\overline{AB}=A\otimes_{R} \overline{B}.
\end{equation}
The functor $\overline{\cdot}$ can be extended to complexes, and defines a functor $\KC^b(\SBim_n)\to \KC^b(\SMod_n)$.
Equation \eqref{eq: bimodule act on module} holds for complexes as well. As a consequence, left tensor multiplication with Rouquier complexes defines a braid group action on $\KC^b(\SMod_n)$.

\begin{lemma}
For $A,B\in \KC^b(\SBim_n)$ one has
$$
\Hom(\overline{A},\overline{B})=\Hom(A,B)\otimes_{R} \k,
$$
where we consider $\Hom(A,B)$ as a right $R$-module. 
\end{lemma}

\begin{proof}
It is sufficient to prove the lemma for $A,B\in \SBim_n$.  Then $A$ and $B$ are free as right $R$-modules, and
$\Hom(A,B)$ is free as a right $R$-module by Lemma \ref{lemma:free homs}, so the result follows.
\end{proof}

\begin{theorem}
\label{thm: serre for smod}
For all $A,B\in \KC^b(\SMod)$ one has 
$$
\Homc(A,B)=\Homc(B,\FT^{-1}\otimes A)^{*},
$$
where in the right hand side we take a linear dual over $\k$.
\end{theorem}

Again, here we regard $\FT^{-1}$ as a complex of Soergel bimodules acting (on the left) on the complex of Soergel modules $A$.

\begin{proof}
We can assume $A=\overline{M}$ and $B=\overline{N}$ for $M,N\in \SBim_n$. Then 
$$
\Homc(M,N)=\Homc(N,\FT^{-1}\otimes M)^{\vee},
$$
where we regard both sides as complexes of (right) $R$-modules. Now.
$$
\Homc(A,B)=\Homc(\overline{M},\overline{N})=\Homc(M,N)\otimes_{R}\k,
$$
$$
\Homc(B,\FT^{-1}\otimes A)=\Homc(\overline{N},\FT^{-1}\otimes \overline{M})=\Homc(\overline{N},\overline{\FT^{-1}\otimes M})=\Homc(N,\FT^{-1}\otimes M)\otimes_{R}\k,
$$
and 
$$
\Homc(B,\FT^{-1}\otimes A)^*=\Homc(N,\FT^{-1}\otimes M)^{\vee}\otimes_{R}\k.
$$
\end{proof}
 
Theorem \ref{thm: serre for smod} was proved earlier in \cite{Bez,MS} by different methods and using the relation between $\SMod$ and the category
$\mathcal{O}$.  

\begin{appendix}
\section{Semiorthogonal decompositions for Coxeter groups}
\label{app:anytype}

Let $W$ be a Coxeter group with the set of reflections $S$. Let $\frakh$ be a realization  of $W$.  
Define $R=\k[\frakh]$,
and $B_s=R\otimes_{R^s}R$ for $s\in S$. The category $\SBim_W$ of Soergel bimodules is the smallest full subcategory of
the category of $R-R$ bimodules containing $R$ and all $B_s$ and closed under direct sums, direct summands, tensor products and 
grading shifts. For $W=S_n$ and $\frakh=\k^n$ we recover the category $\SBim_n$ defined in Section \ref{subsec:SBim}.

Rouquier complexes can be defined in the homotopy category $\KC^b(\SBim_W)$ similarly to Section \ref{subsec:rouquier}:
$$
\rouq_s=[B_s\to R],\ \rouq_s^{-1}=[R\to B_s]
$$
In \cite{Rouquier} Rouquier proved that they satisfy the relations in the braid group associated to $W$. Therefore for any $w\in W$ one can consider  a Rouquier complex $\rouq_w$ corresponding to the positive permutation braid associated to any reduced expression of $w$. It does not depend on the choice of a reduced expression up to homotopy equivalence. 

We define triangulated subcategories $\US_{<w}$ and $\US_{\le w}$ of $\KC^b(\SBim_W)$ generated by the Rouquier complexes $\rouq_v$ with $v<w$ 
(and $v\le w$) in Bruhat order. 

For any $s\in S$, there exists a chain map $\psi_s:\rouq_s\to \rouq_s^{-1}$ such that 
\begin{equation}
\label{eq:general skein}
\Cone[\rouq_s\to \rouq_s^{-1}]=[R\to R].
\end{equation}
As an immediate corollary, we get the following:
\begin{proposition}
\label{prop:positive to negative}
For any $w\in W$ there is a chain map $\psi_w: \rouq_w\to \rouq_{w^{-1}}^{-1}$.
The cone of $\psi_w$ is filtered by $\rouq_{u^{-1}}^{-1}$ for $u<w$ in Bruhat order.
\end{proposition}

In \cite{LW} Libedinsky and Williamson proved a much stronger statement (conjectured by Rouquier in \cite{Rouquier2}, p. 215 before Remark 4.12).

\begin{theorem}[\cite{LW}] \label{th:LW}
\label{th:standard costandard}
If $w\neq v$ then $\Hom(\rouq_v,\rouq^{-1}_{w^{-1}})=0$.
If $w=v$ then $\Hom(\rouq_w,\rouq^{-1}_{w^{-1}})=R$ is generated by the map $\psi_w$.
\end{theorem}

\begin{corollary} \label{cor:LW}
\label{cor:Hom vanishing in all types}
We have $\Hom(\rouq_w,R)=0$ for $w\neq 1$. 
\end{corollary}

In type A this corollary is an easy consequence of Corollary \ref{cor:negative stabilization}. In fact, Theorem \ref{th:LW} also can be deduced:
\begin{prop}
	Theorem \ref{th:LW} follows from Corollary \ref{cor:LW}.
\end{prop}
\begin{proof}
	Assume that $\Hom(F_w,R)=0$ for $w\neq 1$. Note that 
	\[
	\Hom(\rouq_v,\rouq^{-1}_{w^{-1}}) \cong \Hom(\rouq_v \rouq_{w^{-1}}, R).
	\]
	 We induct on the number $\min(l(v), l(w))$. The base case follows from the assumption. Without loss of generality, we may assume $l(v)\leq l(w)$. Let $v=v's$ for a simple reflection $s$ and $l(v')=l(v)-1$. If $l(ws)>l(w)$ then $w\neq v$ and $ws\neq v'$, so we have
	\[
	\Hom(\rouq_v \rouq_{w^{-1}}, R) = \Hom(\rouq_{v'} \rouq_{s w^{-1}}, R) = \Hom(\rouq_{v'} \rouq_{(ws)^{-1}}, R) \cong 0,
	\]
	where the last equality follows from the induction hypothesis. So we assume $l(ws)<l(w)$. Let $w=w' s$. 
	The map $\psi_s:F_s\to F_s^{-1}$ induces a map
	\[
	\Hom(\rouq_{v'}, \rouq_{{w'}^{-1}}^{-1}) = \Hom(\rouq_{v'} \rouq_s^{-1}, \rouq_{{w'}^{-1}}^{-1} \rouq_s^{-1}) \to \Hom(\rouq_{v'} \rouq_s, \rouq_{{w'}^{-1}}^{-1} \rouq_s^{-1}) = \Hom(\rouq_v, \rouq_{w^{-1}}^{-1}),
	\]
	whose cone is filtered by $\Hom(\rouq_{v'}, \rouq_{w^{-1}}^{-1})$, which vanishes by the induction 
	hypothesis since $l(v')<l(v)\leq l(w)$. So we are reduced to the statement for the pair $v', w'$.
\end{proof}

We use Theorem \ref{th:standard costandard} to deduce a very important fact about Rouquier complexes which does not appear to be explicitly stated in the literature.

\begin{theorem}
\label{th:triangular hom}
We have $\Hom(\rouq_w,\rouq_v)=0$ unless $w\leq v$ in Bruhat order.
\end{theorem}

\begin{proof}
By Proposition \ref{prop:positive to negative} $\rouq_v$ is homotopy equivalent to a complex filtered by $\rouq_{u^{-1}}^{-1}$ 
with $u\le v$. Therefore $\Hom(\rouq_w,\rouq_v)=0$ unless $\Hom(\rouq_w,\rouq_{u^{-1}}^{-1})\neq 0$ for some $u\le v$.
But by Theorem \ref{th:standard costandard} this is possible only if $u=w$, and hence $w\le v$. 
\end{proof}

\begin{corollary}
\label{cor:adding w}
For all $w$ we have semiorthogonal decompositions $\US_{\le w}=\langle \US_{<w}, \rouq_w\rangle$ and 
$\US_{\le w}=\langle \rouq_{w^{-1}}^{-1}, \US_{<w}\rangle$.
\end{corollary}

\begin{proof}
By Proposition \ref{prop:positive to negative} the category $\US_{\le w}$ is generated by $\US_{<w}$ and $\rouq_w$, or, equivalently, 
by $\US_{<w}$ and $\rouq_{w^{-1}}^{-1}$ (since $\Cone[\rouq_w\to \rouq_{w^{-1}}^{-1}]\in \US_{<w}$). Now for all $u<w$ we have
$\Hom(\rouq_w,\rouq_u)=0$ by Theorem \ref{th:triangular hom} and $\Hom(\rouq_{u},\rouq_{w^{-1}}^{-1})=0$ by Theorem \ref{th:standard costandard}.
\end{proof}

\begin{corollary}
\label{cor: left right adjoints Bruhat}
If $W$ is a finite Coxeter group then for all $w\in W$ the inclusion $\US_{\leq w}\hookrightarrow \KC^b(\SBim_W)$
has both left and right adjoints.
\end{corollary}

\begin{proof}
Fix an arbitrary total order $\prec$ on W refining the Bruhat order, let $w_0$ be the longest element in $W$. Then for all $w$ we have a chain
$$
w=w^{(1)}\prec w^{(2)}\prec \ldots \prec w^{(k)}=w_0.
$$
Similarly to Corollary \ref{cor:adding w}, the inclusions $\US_{\leq w^{(i)}}\hookrightarrow \US_{\leq w^{(i+1)}}$ have both left and right adjoints, and by combining these we get adjoints to the inclusion
$$
\US_{\leq w^{(i)}}\hookrightarrow \US_{\leq w_0}=\KC^b(\SBim_W).
$$
\end{proof}

If $W'$ is a parabolic subgroup of $W$, we can consider the category of Soergel bimodules $\SBim_{W'}$ associated to the same realization $\frakh$.

\begin{corollary}
Let $W$ be a finite Coxeter group and  $W'$ a parabolic subgroup. Then the inclusion $\KC^b(\SBim_{W'})\to \KC^b(\SBim_W)$
has both left and right adjoints.
\end{corollary}

\begin{proof}
We have $\KC^b(\SBim_{W'})=\US_{\leq w}$ where $w$ is the longest element of $W'$.
\end{proof}

Note that in type A this gives an alternative construction of adjoints to inclusions of $\SBim_{n,1,\ldots,1}$ in $\SBim_{m}$.
However, it seems that the direct construction of adjoints in Section \ref{sec:semiortho} is easier to work with than the induction on Bruhat graph as in Corollary \ref{cor: left right adjoints Bruhat}.
\end{appendix}

\bibliographystyle{alpha}

\bibliography{refs}

\end{document}